\documentclass[11pt]{article}
\usepackage[utf8]{inputenc}
\usepackage{multirow}  
\usepackage{soul}      
\usepackage{a4}        
\usepackage{amssymb,amsmath,amsthm} 
\usepackage{amsfonts}  
\usepackage{hyperref}  
\usepackage{mathtools} 
\usepackage{graphicx}  
\usepackage{subcaption}
\usepackage{float}    
\usepackage{cleveref} 
\usepackage{enumitem} 
\usepackage{multicol}

\newtheorem{theorem}{Theorem}[section]

\newtheorem{lemma} [theorem]{Lemma}

\theoremstyle{remark}
\newtheorem{remark}[theorem]{Remark}
\newtheorem{example}[theorem]{Example}
\newcommand{\R}{\mathbb{R}}

\newcommand{\N}{\mathbb{N}}

\newcommand{\x}{\mathbf{x}}
\newcommand{\y}{\mathbf{y}}
\newcommand{\I}{\mathcal{K}}

\newcommand{\bigO}{O}

\title{Rectangular polar quadrature in 1D and its error analysis}
\author{Krishna Yamanappa Poojara \footnote{Computing and Mathematical Sciences, Caltech, Pasadena, CA 91125, USA, kyp6174@caltech.edu}, \and Sabhrant Sachan\footnote{Computing and Mathematical Sciences, Caltech, Pasadena, CA 91125, USA, ssachan@caltech.edu}, \and Ambuj Pandey \footnote{Indian Institute of Science Education and Research Bhopal (IISER Bhopal), ambuj@iiserb.ac.in}}
\date{}

\begin{document}
	\maketitle
	\begin{center}
		Preliminary draft
	\end{center}
	\begin{abstract}
%

This paper presents a one-dimensional analog of the Rectangular-Polar (RP) integration strategy \cite{bruno2020chebyshev} and its convergence analysis for weakly singular convolution integrals. The key idea of this method is to break the whole integral into integral over non-overlapping patches (subdomains) and achieve convergence by increasing the number of patches while approximating the integral on patches accurately using a fixed number of quadrature points. The non-singular integrals are approximated to high-order using Fej\'er first quadrature, and a specialized integration strategy is designed and incorporated for singular integrals where the kernel singularity is resolved by mean of Polynomial Change of Variable (PCV).  We prove that for high-order convergence, it is essential to compute integral weights accurately, and the method's convergence rate depends critically on the degree of the PCV and the singularity of the integral kernel. Specifically, for kernels of the form $|x-y|^{-\alpha}(\alpha>-1)$, the method achieves high-order convergence if and only if $p(1-\alpha) \in \mathbb{N}$. This relationship, first observed numerically in \cite{bruno_sachan2024numerical}, highlights the importance of choosing an appropriate degree $p$ for the PCV. 
A new error estimate in a framework where convergence is achieved by reducing the approximating domain size while keeping the number of discretization points fixed is derived for the Fej\'er first quadrature, decay rate of Chebyshev coefficients, and the error in approximating continuous Chebyshev coefficients with discrete ones.
Numerical experiments corroborate the theoretical findings, showcasing the effectiveness of the RP strategy for accurately solving singular integral equations. As an application of the method, numerical solutions of the surface scattering problem in the two dimensions are computed for complex domains, and the algorithm's efficacy is demonstrated for large-scale problems.
\end{abstract}
	
	\section{Introduction}

Recently, an efficient high-order Rectangular-Polar (RP) integration strategy was introduced in \cite{bruno2020chebyshev} for the approximation of weakly singular convolution integrals which appear in the context of surface scattering problems in three dimensions. In \cite{bruno2020chebyshev}, the authors demonstrated the spectral accuracy of the RP integration strategy through several numerical experiments and simulated scattering phenomena on complex geometries relevant to real-world applications. Due to its ability to handle singular kernels with high-order accuracy, even for complex geometric domains, the RP strategy has been successfully applied to various practical problems in science and engineering \cite{zhang2022hyper,ng2023acoustic,garza2022fast,faria2021general,hu2021chebyshev}. The hallmark of the RP integration scheme's high accuracy lies in its design and the specialized quadrature rules it employs for the evaluation of singular integrals. These rules leverage a Polynomial Change of Variable (PCV), which clusters nearly half of the integration nodes near the singular points, thereby regularizing the kernel. While the high-order convergence of the RP quadrature scheme has been demonstrated in numerous applications \cite{hu2021chebyshev,bruno2024direct,ng2023acoustic}, rigorous convergence analysis of the method is, to the best of our knowledge, not yet reported in the literature. In this paper, we present a rigorous quadrature analysis of a one-dimensional analog of the RP integration scheme for the evaluation of convolution integral operators of the form:
\begin{equation}\label{conv_eqn}
	\I[u](x) := \int_{\Omega} g_{\alpha}\left(|x-y| \right) u(y) \, dy, \quad x \in \Omega,
\end{equation}
where $\Omega$ is a closed, bounded interval in $\mathbb{R}$, and $u$ denotes an integral density assumed to belong to the function space $X^m(\Omega) = C^m(\Omega) \cap C_{\text{pw}}^{m+2}(\Omega)$, implying $u$ is $m$-times continuously differentiable with piecewise continuous $(m+2)$-th derivatives. The kernel $g_{\alpha}\left(|x-y|\right)$ takes the form:
\begin{equation} \label{eq_kernel}
	g_{\alpha}\left(|x-y|\right) :=
	\begin{cases}
		|x-y|^{-\alpha}, & 0 < \alpha < 1, \\
		\log|x-y|, & \alpha = 0.
	\end{cases}
\end{equation}

This analysis is critical for selecting key algorithmic parameters that influence the high-order accuracy of the methodology and its extension to higher dimensions. For example, in \cite{bruno_sachan2024numerical}, the authors computationally observed that, unlike the logarithmic kernel, the RP strategy does not exhibit high-order convergence for kernels of the form $ |x-y|^{-\alpha} $ for all polynomial degrees in the PCV. Specifically, they noted that for kernels of the form $ g_{\alpha}(x, y) = |x-y|^{-\alpha} $ with $ 0 < \alpha < 1 $, the convergence rate depends on the degree $p$ of the polynomial in the PCV, achieving high-order convergence only for specific values of $p$. Empirical evidence suggests that high-order convergence occurs when $ p(1-\alpha) \in \mathbb{N} $. This raises an important question: \textit{for a given $\alpha$ and density regularity $m$, what values of $p$ ensure high-order convergence of the RP quadrature}? We answer this question (see \Cref{error}, \Cref{Thm_quad_error_1d}, and \Cref{Thm_Patches_Varying}) and rigorously prove that the RP method achieves high-order convergence if and only if $ p(1-\alpha) \in \mathbb{N} $. In such cases, the precise order of convergence is $ m+2-\alpha $. When $ p(1-\alpha) \notin \mathbb{N} $, the convergence order is limited to $\min\{m+2-\alpha, 2p(1-\alpha)\}$. For infinitely smooth $u$, the scheme achieves exponential convergence if $ p(1-\alpha) \in \mathbb{N} $, while otherwise, the convergence rate is restricted to $ 2p(1-\alpha) $. These theoretical results are corroborated by extensive computational experiments, showing close agreement between theoretical and computational convergence rates.

To improve computational efficiency, the RP strategy \cite{bruno2020chebyshev} divides the integration domain into non-overlapping subdomains (patches), approximating the integral on each patch with a high-order quadrature rule using a fixed number of nodes. High-order convergence is achieved by increasing the number of patches while keeping the number of nodes on each patch constant. However, a rigorous error analysis of this approach requires results on the following for functions $f \in X^m[a, b]$ as a function of the interval length $h = b-a$:
\begin{enumerate}
	\item Error in Fej\'{e}r first quadrature.
	\item Decay rate of Chebyshev coefficients of $f$.
	\item Truncation error in approximating $f$ by its Chebyshev series.
	\item Error in approximating continuous Chebyshev coefficients with discrete coefficients.
\end{enumerate}
While these results are well-established for fixed $h$ as $n \to \infty$, they are not explicitly available for the case where convergence is achieved by shrinking $h$ while keeping $n$ fixed. In this paper, we derive explicit error estimates for both cases, making our results applicable to various computational techniques where convergence is achieved by domain decomposition, such as $hp$-Finite Element and Spectral Element Methods. As mentioned in \cite[cf. Section 5.1]{bruno2020chebyshev}, our analysis also demonstrates the necessity of accurate singular integral weights for high-order convergence. If singular weights are computed inaccurately, the asymptotic convergence rate drops to $O(h)$ for logarithmic kernels and $O(h^{1-\alpha})$ for $\alpha$-kernels. However, numerical experiments indicate that 24–32 nodes per patch with appropriate $p$ are sufficient for accurate weight computation, mitigating this issue. Finally, we make an important observation regarding singularity treatment. Breaking the singular integral at the singularity point improves the convergence rate by two orders compared to treating it without subdivision. This insight underscores the importance of specialized singularity handling in achieving high accuracy with the RP strategy. 

The proposed one-dimensional RP-Integration scheme offers additional computational advantages over a higher-dimensional version. For instance, unlike the three-dimensional approach \cite{bruno2020chebyshev} where singular integral weights require $O(P_{\text{near}}N)$ memory, in the one dimension, it needs only $O(1)$ memory for the weight storage, where $N$ is the number of unknowns and $P_{\text{near}}$ is the number of near singular patch. 	Thus, the proposed one-dimensional RP integration scheme can efficiently evaluate convolution integrals of the form \ref{conv_eqn} with high-order accuracy. 
For the N-point discretization problem, our convolution computation needs $O(N^2)$ operations. However, this process can be significantly sped up using existing fast summation methods such as the Fast Multipole Method \cite{martinsson2007accelerated}, and the sum of exponential approximation \cite{zhang2021fast} to name a few.

The	proposed method can be straightforwardly adopted as a Nystr\"{o}m solver for an accurate numerical solution of integral equations in one dimension. For instance, in \cite{bruno2024direct} the RP strategy is incorporated for the solution of the second kind Fredholm equation over the boundary of square domain. Moreover, the applicability of the RP strategy can be easily extended for the boundary of complicated domains and to illustrate this, in \Cref{sec:scattering}, we have implemented the proposed method for the solution of two-dimensional surface scattering by acoustic waves. We have simulated the aforementioned scattering problem by employing the proposed method to the second kind equivalent integral equation and the resulting linear system is solved iteratively by means of the generalized minimal residual method (GMRES). Unlike the Finite Element and Finite Difference method, our solver is dispersion less and capable of solving a large-scale problem accurately.

The rest of the paper is organized as follows. In \Cref{method} we present the details of the proposed quadrature which comprises of decomposition of the domain into non-overlapping patches, treatment of regular, singular, and near-singular integrals, and the implementation of the change of variable. \Cref{section_decay_of_coefficient} presents new error bounds on Chebyshev coefficients related to Chebyshev polynomials of the first kind in terms of varying patch length and index of the coefficient, from which we have also derived a novel error bound on Fej\'er quadrature. \Cref{error} presents the error analysis of the proposed numerical integration method.  \Cref{numerics} demonstrates a variety of numerical examples, which align with the error analysis.  To demonstrate the applicability of our method, in \Cref{sec:scattering} we have also included simulation results for the two-dimensional surface scattering problems. We concluded the paper with comments on possible directions for future work.

\section{Methodology}\label{method}
	
	In this section, we describe the proposed integration scheme for an accurate evaluation of integral operator 
	\begin{equation}\label{conv_eqn2}
		\I[u](x) := \int\limits_{\Omega} g_{\alpha}\left(|x-y | \right)u(y) dy,  \ \text{for} \  x\in \Omega,
	\end{equation}
	where  $\Omega$ is a closed bounded interval in $\mathbb{R}$, and the integral density  $u \in X^m(\Omega), m\ge 0$. We also assume that the integral kernel $g_{\alpha}\left(|x-y |\right)$ is a weakly singular, that is, $g_{\alpha}(|x-y|)$ is continuous for all $x,y \in \Omega$, $x \ne y$, and there exist a positive constant $C$ such that
	\begin{equation} \label{weak_kernel}
		\left|g_{\alpha}\left(|x-y |\right)\right| \le C|x-y|^{-\alpha},\quad x,y \in \Omega, \, x \ne y, 
	\end{equation}
	for $0<\alpha<1$. Owing to the kernel singularity in equation \eqref{conv_eqn2} a straightforward application of the high-order quadrature rule exhibits low order convergence, and hence a specialized quadrature rule must be designed and used if high-order accuracy is desired. In order to deal with the singular character of the kernel efficiently, we utilize a strategy based on local parametrization that we explain in what follows. To this end, we start by representing  $\Omega$ as a union of $P$ patches namely $\Omega_{1}, \Omega_{2},...,\Omega_{P}$
	such that 
	\begin{equation}
    \label{patches_condition}
    \Omega = \bigcup\limits_{\ell=1}^{P} \overline{\Omega}_{\ell},\  \text{ and } \Omega_{\ell}\cap \Omega_{q} =\emptyset \text{ for } \ell\neq q,
    \end{equation}
    where
	each $\overline{\Omega}_{\ell}$ (closure of $\Omega_{\ell}$) is an image of the closed set $[-1,1]$ via a smooth invertible  parametrization $\xi_{\ell}$. Using this patching decomposition and parametrization, integral \eqref{conv_eqn2} can be re-written as 
	\begin{equation}\label{conv_patch_eqn}
		\I[u](x) = \sum\limits_{\ell=1}^{P} \I_{\ell}[u](x) , 
	\end{equation}
	where 
	\begin{equation}\label{int_on_patch}
		\I_{\ell}[u](x) = \int\limits_{-1}^{1} g_{\alpha}(|x-\xi_{\ell}(t)|) u(\xi_{\ell}(t)) J_{\ell}(t) dt,
	\end{equation}	
    and $J_{\ell}(t)$ denote the Jacobian of parametrization $\xi_{\ell}$. Note that a high-order approximation of integral 	\eqref{int_on_patch} depends on the position of target point $x$ with respect to the integration patch $\Omega_{\ell}$.  For instance, if the target point $x$ lies in the integration patch $\Omega_{\ell}$ or very close to $\Omega_{\ell}$ then the integrand in \eqref{int_on_patch} becomes singular or near singular respectively.  Thus, depending upon the location of $x$ relative to the integration patch $\Omega_{\ell}$, we partition the set of integral in \eqref{conv_patch_eqn} into three different classes, namely, regular integral, singular integral and near singular integral. \medskip
	
	\noindent
	In the case of regular integral, target point $x$ is sufficiently away from the integration patch $\Omega_{\ell}$, so the integrand in \eqref{int_on_patch} becomes smooth (as the integral kernel is smooth),  and high-order accuracy can be achieved by utilizing any of the classical high-order quadrature rules. For instance, in our implementations, to evaluate the integral \eqref{int_on_patch} we have used the FF-Rule, discussed in \cite{davis2014methods} and achieved high-order convergence. To be precise, if $x$ is sufficiently away from $\Omega_\ell$, then the discrete operator 
	
	\begin{equation}\label{eq_regular_int}
		\mathcal{K}^{n}_{\ell,\textrm{reg}}[u](x)=\sum\limits_{i=0}^{n-1} w_{i} g_{\alpha}(|x-\xi_{\ell}(t_{i})|)u(\xi_{\ell}(t_{i}))J_{\ell}(t_{i}),			
	\end{equation}
yields high-order approximation to the equation \eqref{int_on_patch}, where $t_{i} = \cos{\left(\pi\frac{2i+1}{2n}\right)}$ denotes the open Chebyshev nodes in the interval $[-1,1]$ and ${w}_{i}$ denotes the weights for the FF-Rule \cite{davis2014methods}, given by
\begin{align}\label{Chebyshev_weights}
	w_{i} = \frac{2}{n} \left(1-2 \sum\limits_{k=0}^{\lfloor \frac{n}{2} \rfloor} \frac{\cos(k v_{2i+1} )}{4k^{2}-1} \right),\,\, v_{2i+1} = \frac{(2i+1)\pi}{n}.
\end{align}
We discuss the convergence of FF-Rule in \Cref{theorem_quad_convergence}, which illustrates order of convergence of this quadrature in terms of the decay rate of  Chebyshev coefficients of the function. The central focus of our presentation lies in the treatment of singular integral and it's convergence analysis, which arises when the target point $x$ lies in the integration patch $\Omega_{\ell}$. To circumvent the kernel singularity in $g_\alpha$ at $x=y$, we utilize local Chebyshev expansion of the integral density and a local change of variable that we describe in what follows. The truncated Chebyshev expansion of integral density $\phi_{\ell}(t) = u(\xi_{\ell}(t))J_{\ell}(t)$ in the patch $\Omega_\ell$ can be obtained as
\begin{equation} \label{trunc_cheby_expan}
	\phi_{\ell}(t) \approx \sum\limits_{k=0}^{n-1} \tilde{c}_{k}^{\ell}T_{k}(t),
\end{equation}
where $T_{k}$ denotes the Chebyshev polynomial of the first kind of degree $k$ \cite{mason2002chebyshev}, and $\tilde{c}_{k}^{\ell}$ is the discrete Chebyshev coefficient of $\phi_\ell$ that can be computed as
\begin{equation}
	\tilde{c}_{k}^{\ell} = \frac{\gamma_{k}}{n}\sum\limits_{i=0}^{n-1} \phi_{\ell}(t_{i}) T_{k} (t_i),  ~ \text{with}~\gamma_{k} = \left\{
	\begin{array}{ll}
			1, & \mbox{if } k=0 \\
			2, & \mbox{otherwise}.
	\end{array}
    \right.
\end{equation}
Note that, the coefficients $\tilde{c}_{k}^{\ell}$, for all $k=0,1,\cdots,n-1,$ and for all $\ell=1,2,\cdots,P,$ can be computed in $O(nP \log n)$ operations using the FFT.  Now using discrete Chebyshev series (\ref{trunc_cheby_expan}), integral \eqref{int_on_patch} can be approximated as 
\begin{equation}\label{eq_Cheby_expn_operator}
\I_{\ell}[u](x) \approx \I_{\ell}^{n}[u](x) = \sum\limits_{k=0}^{n-1} \tilde{c}_{k}^{\ell}\beta_{k}^{\ell}(x), \text{ where }\beta_{k}^{\ell}(x) = \int\limits_{-1}^{1} g_{\alpha}(|x-\xi_{\ell}(t)|)~T_{k}(t) ~dt. 
\end{equation}
Note that the integrand in $\beta_{k}^{\ell}(x)$ is singular at $x=\xi_{\ell}(t)$ and hence special care is required for it's accurate approximation. To compute $\beta_{k}^{\ell}(x)$ with high degree of precision, we have employed a change of variable technique similar to those utilized in \cite{bruno2020chebyshev}. Therefore, we now employ the change of variable $\psi_{p}$ which smoothens the integrand at the edges as discussed in \cite{bruno2020chebyshev, kress1990nystrom} to handle the singularity in the integral $\beta_{k}^{\ell}(x)$, which is defined  below for $p \ge 2$,
\begin{equation}\label{w_p(t)}
	\psi_{p}(t) = 2 \frac{\left[v_{p}(t)\right]^{p}}{\left[v_{p}(t)\right]^{p} + \left[v_{p}(- t)\right]^{p}},
\end{equation}
and
\begin{equation}\label{v_p(t)}
	v_{p}(t) =
	\left(\frac{1}{2} - \frac{1}{p} \right) t^{3} + \frac{t}{p} + \frac{1}{2}, ~-1 \leq t \leq 1.
\end{equation}
There are two approaches to apply this change of variable: one to overcome the singularity within the interior of the domain and the other to address the singularity at the edges. Upon implementing both approaches we observed a peculiar behavior in the convergence order, which we demonstrate in the table below.
\begin{figure}[!htb]
	\centering
	\hspace{1.5cm}\begin{minipage}{.3\linewidth}
		\centering
		\includegraphics[width=1.\linewidth]{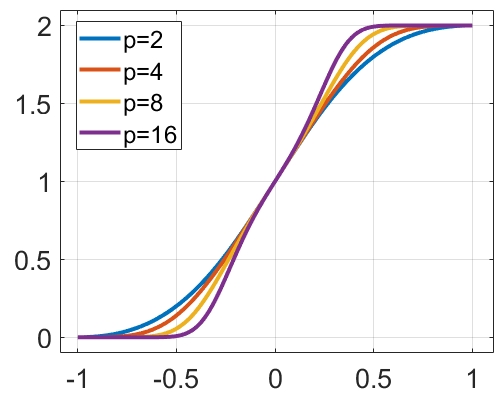}	
			\label{Plot_wp(t)}
		\end{minipage}\hfill
	\begin{minipage}{0.6\linewidth}
		\centering
		\scalebox{.8}{
		\begin{tabular}{c| c| c| c| c}
			\hline \hline 
			\multirow{2}{*}{n} & 
			\multicolumn{2}{c|}{boundary} & 
					\multicolumn{2}{c}{interior} \\
					\cline{2-5}
					& error & order & error & order \\ 
					\hline
					4 & 2.46e-02 & - & 2.75e-01 & -\\
					8 & 2.98e-05 & 9.69 & 1.33e-02 & 4.37\\
					16 & 1.65e-07 & 7.50 & 3.02e-05 & 8.78\\
					32 & 7.27e-14 & 21.1 & 5.58e-08 & 9.08\\
					64 & - & - & 7.10e-10 & 6.30\\
					128 & - & - & 9.28e-12 & 6.26\\
					256 & - & - & 1.19e-13 & 6.29\\
					512 & - &  & 3.00e-15 & 5.31\\
					1024 & - & - & 2.00e-16 & 3.91\\
					\hline \hline
			\end{tabular} } 
		\end{minipage}
		\caption{The change of variable function $\psi_{p}$ is plotted for various values of $p$, we can see the higher order derivatives of $\psi_{p}$ vanishes at the edges as $p$ increases. The table in right demonstrates behavior of the polynomial change of variable \cite[equation (40)]{bruno2020chebyshev}, which is implemented for the function $|t|^{-0.1}$ on $[-1,1]$ at the singularity $t=0$ on boundary and interior with $p=7$. We refer \cite[Theorem 2]{kutz1984asymptotic} for theoretical justification.}
	\end{figure}
	\noindent
As illustrated above, the change of variable method addressing the singularity at the edges shows superior convergence compared to the one focusing on the interior singularity. Consequently, we split the integral in equation \eqref{eq_Cheby_expn_operator} at the singular point $t_{x} \coloneqq \xi_{\ell}^{-1}(x)$ into two parts: $\beta_{k}^{\ell}(x) = \beta_{k,L}^{\ell}(x) + \beta_{k,R}^{\ell}(x)$, where
\begin{equation}\label{eq_weights_betaLR}
\beta_{k,L}^{\ell}(x)=\int\limits_{-1}^{t_{x}} g_{\alpha}(|x-\xi_{\ell}(t)|) T_{k}(t) dt \text{, and } \beta_{k,R}^{\ell}(x)=\int\limits_{t_x}^{1} g_{\alpha}(|x-\xi_{\ell}(t)|) T_{k}(t) d t.    
\end{equation}			
After utilizing the change of variable \eqref{w_p(t)}, we approximate $\beta_{k,L}^{\ell}$, $\beta_{k,R}^{\ell}$ by FF-Rule, and denote the approximations as $\tilde{\beta}_{k,L}^{\ell}$ and $\tilde{\beta}_{k,R}^{\ell}$ respectively. Finally, we conclude our singular integration scheme by introducing the singular integral operator $\I_{\ell,\textrm{sing}}^{n}[u](x)$, 
\begin{equation}\label{sing_quad}
	\I_{\ell}^{n}[u](x) \approx \I_{\ell,\textrm{sing}}^{n}[u](x) = \sum\limits_{k=0}^{n-1} \tilde{c}_{k}^{\ell}(\phi_{\ell})\tilde{\beta}_{k}^{\ell}(x), \quad \tilde{\beta}_{k}^{\ell} = \tilde{\beta}_{k,L}^{\ell}+\tilde{\beta}_{k,R}^{\ell}.
\end{equation} 	
The only case which remains to be dealt with is near singular integration, in which the target point is close to the integration patch such that a direct application of the FF-Rule exhibits poor convergence due to the presence of the weakly singular kernel. Although the kernel remains finite in this case, it can blow up to significantly larger values, creating numerical challenges. We follow two fold strategy to overcome the near singularity phenomenon. First, we  project $t_{x}$ (see \eqref{eq_weights_betaLR}) to the nearest edge of the parametric space corresponding to the integration patch. Then, we follow the singular integration scheme discussed above for the projected point.

 \section{Chebyshev Coefficients: Decay with Interval Size}\label{section_decay_of_coefficient}
 Throughout this paper, the symbol $C$ stands for a positive constant
taking on different values on different occurrences.
In this section, we will discuss the decay of the Chebyshev coefficients which incorporates length of the interval as it play a major role in deriving convergence rates of our methodology. In literature, many authors have discussed the convergence of Chebyshev coefficients \cite{xiang2015error,xiang2010error,xiang2013convergence,trefethen2008gauss}, however, in best of our knowledge, the decay rate in terms of the length of the interval are not reported. We now investigate the convergence rates of coefficients with size of the patch and index. Throughout this section, let $\Omega = [a,b]$, $h=b-a$ and define $c_k$ as the Chebyshev coefficients of the function $\tilde{u}$, where 
\begin{equation}\label{eq_ul_def}
\tilde{u}(t)=u\left(\xi(t)\right),\ \xi(t) = \frac{h}{2}t+\frac{a+b}{2}, \ -1\le t\le 1.
\end{equation} 
The function $\tilde{u}$ is simply a re-parametrization of $u$ on the interval $\Omega$ to $[-1,1]$.  We aim to estimate $c_{k}$ in terms of the patch length $h$ and coefficient index $k$, to this end, we begin with a preliminary lemma showing the behavior of the coefficients as $h$ tends to zero.
\begin{lemma}
     Let $u$ be a continuous function on $\Omega$, $x_0 = \frac{a+b}{2}$ and $k \geq 1$. If $a,b$ are bounded as $h$ tends to zero, then the Chebyshev coefficients $c_k$ of $\tilde{u}$ satisfy
    \begin{equation*}
    c_0\rightarrow u\left(x_0\right), \,\, c_k \rightarrow 0, \text { as } h \rightarrow 0.
    \end{equation*}
\end{lemma}  
\begin{proof}
For $k \geq 1$,
\[
c_k=\frac{2}{\pi} \int_{-1}^1 \tilde{u}(t) \frac{T_k(t)}{\sqrt{1-t^2}} d t=\frac{2}{\pi} \int_0^\pi \tilde{u}(\cos \theta) \cos (k \theta) d \theta .\]
Using Dominated convergence theorem,
\[\lim _{h \rightarrow 0} c_k=\frac{2}{\pi} \int_0^\pi \lim _{h \rightarrow 0} \tilde{u}(\cos \theta) \cos (k \theta) d \theta=\frac{2}{\pi}u(x_0) \int_0^\pi \cos (k \theta) d \theta=0 .\]
Similarly, the result follows for $k=0$. 
\end{proof}
\noindent
To understand how fast Chebyshev coefficients $c_k$ converge to zero for $k\ge 1$, we use integration by parts. We obtain
	\begin{equation}\label{Eq_Cheby_Coeff_intPart_one}       
		c_k = \frac{1}{\pi}\int_{-\pi}^{\pi}\tilde{u}(\cos{\theta})\cos(k\theta)d\theta = \frac{h}{2k}\frac{1}{\pi}\int_{-\pi}^{\pi}\tilde{u}'(\cos{\theta})\sin(\theta)\sin{(k\theta)}d\theta.
	\end{equation}
	To simplify calculations, define 
	\begin{equation}\label{Ck_int_by_parts_term}
		I_{M}[\tilde{u}^{(i)},j,k] = \int_{-\pi}^{\pi} \tilde{u}^{(i)}(\cos{\theta})\sin^{j}(\theta)\cos^{i-j}(\theta)\zeta_{M}(k\theta)d\theta,  
   \end{equation}
	where $i,M$ are non-negative integers, and $j$ is an integer such that $j\le i$. If $j<0$, then define $I_{M}[\tilde{u}^{(i)},j,k] = 0$. Here $M$ corresponds to the number of times integration by parts is performed and $\zeta_M(\theta) = \cos{\theta}$ if $M$ is even, else $\zeta_M(\theta) = \sin{\theta}$; $\tilde{u}^{\scriptscriptstyle (i)}$ corresponds to $i$-th derivative of $\tilde{u}$, where $i$ depends on the regularity of $\tilde{u}$. Before formulating an expression for the coefficients $c_{k}$, we describe a recursive relation which plays a crucial rule in obtaining a bound for $c_k$. 
\begin{lemma}\label{lemma_expression_Ck_induction}
If $u\in X^{m}(\Omega)$ then for integers $i$, $j$, $M$ with $0\leq j \le i\leq m$, $M\ge 0$ and $k\in \N$, $I_{M}[\tilde{u}^{(i)},j,k]$ satisfies the following recursive relation
    \begin{equation}\label{eq_rec_reln}
		I_{M}[\tilde{u}^{(i)},j,k] =  \frac{(-1)^M}{k}\left(-jI_{M+1}[\tilde{u}^{(i)},j-1,k]+(i-j)I_{M+1}[\tilde{u}^{(i)},j+1,k]+ \frac{h}{2}I_{M+1}[\tilde{u}^{(i+1)},j+1,k]\right). 
	\end{equation} 
    Additionally, if $u\in C^{m+2}(\Omega)$ then equation \eqref{eq_rec_reln} holds for $0\leq j
    \leq i\leq m+1$.
\end{lemma}
\begin{proof}
Using integration by parts, and the relation 
    \begin{equation*}
        \int \zeta_M(k\theta)d\theta = \frac{(-1)^M}{k}\zeta_{M+1}(k\theta),
    \end{equation*}
  for $0\leq i \leq m-1$, $I_{M}[\tilde{u}^{(i)},j,k]$ can be written as
 \begin{align} 
    I_{M}[\tilde{u}^{(i)},j,k] &= \int_{-\pi}^{\pi} \tilde{u}^{(i)}(\cos{\theta})\sin^{j}(\theta)\cos^{i-j}(\theta)\zeta_{M}(k\theta)d\theta \label{int_by_parts_step}\\
    &= \frac{(-1)^{M}}{k} \frac{h}{2}\int_{-\pi}^{\pi} \tilde{u}^{(i+1)}(\cos{\theta})\sin^{j+1}(\theta)\cos^{i-j}(\theta) \zeta_{M+1}(k\theta)d\theta \nonumber\\
    &\quad +\frac{(-1)^{M+1} j}{k}\int_{-\pi}^{\pi}\tilde{u}^{(i)}(\cos{\theta})\sin^{j-1}(\theta)\cos^{i-(j-1)}(\theta)  \zeta_{M+1}(k\theta)d\theta \nonumber\\
    &\quad +\frac{(-1)^{M}}{k}(i-j)\int_{-\pi}^{\pi} \tilde{u}^{(i)}(\cos{\theta})\sin^{j+1}(\theta)\cos^{i-(j+1)}(\theta) \zeta_{M+1}(k\theta)d\theta. 
\end{align}
Hence, for $u\in X^{m}(\Omega)$ we have the following recurrence relation for $1 \le i \leq m-1$, 
\begin{equation}\label{eq_rec_reln_m-1}
I_{M}[\tilde{u}^{(i)},j,k] =  \frac{(-1)^M}{k}\left(-jI_{M+1}[\tilde{u}^{(i)},j-1,k]+(i-j)I_{M+1}[\tilde{u}^{(i)},j+1,k]+ \frac{h}{2}I_{M+1}[\tilde{u}^{(i+1)},j+1,k]\right). 
\end{equation} 
\noindent
Since $u^{(m)}$ is not  differentiable in $\Omega$ and $\tilde{u}(\cos{(\cdot)})\in C^{m+2}_{\mathrm{pw}}[-\pi,\pi]$, there exist discontinuities $\{s_{\ell}\}_{\ell=0}^{n_d}$ of $(m+1)$-th derivative of $\tilde{u}(\cos{\theta})$ for some $n_{d}\in \N$ such that $[-\pi,\pi] = \cup_{\ell=1}^{n_d} [s_{\ell-1},s_\ell]$. Therefore, $\tilde{u}(\cos{(\cdot)})\in C^{m+2}(s_{\ell-1},s_{\ell})$ for each $\ell$. Using this decomposition, an application of integration by parts one more time gives
     \begin{align}
     I_M[\tilde{u}^{(m)},j,k] &= \int_{-\pi}^{\pi} \tilde{u}^{(m)}(\cos{\theta})\sin^{j}(\theta)\cos^{i-j}(\theta)\zeta_{m}(k\theta)d\theta \nonumber\\
     &= \frac{(-1)^m}{k}\sum_{\ell=1}^{n_d}\left[ \left(\tilde{u}^{(m)}(\cos{\theta})\sin^{j}(\theta)\cos^{i-j}(\theta) \zeta_{m+1}(k\theta)\right)_{s_{\ell-1}}^{s_\ell} \right. \nonumber\\
     &\left.\hspace{1cm} -\int_{s_{\ell-1}}^{s_{\ell}} \left(\tilde{u}^{(m)}(\cos{\theta})\sin^{j}(\theta)\cos^{i-j}(\theta)\right)'\zeta_{m+1}(k\theta)d\theta \right]. \label{intbyparts_discont}
    \end{align}
    Using continuity of $u^{(m)}$, the first term in \eqref{intbyparts_discont} is zero. Since the discontinuities of $\tilde{u}^{(m+1)}$ are finitely many, the second term in \eqref{intbyparts_discont} can be simplified as 
    \begin{align*}
        \sum_{\ell=1}^{n_d} \int_{s_{\ell-1}}^{s_{\ell}} \left(\tilde{u}^{(m)}(\cos{\theta})\sin^{j}(\theta)\cos^{i-j}(\theta)\right)'\zeta_{m+1}(k\theta)d\theta =  \int_{-\pi}^{\pi} \left(\tilde{u}^{(m)}(\cos{\theta})\sin^{j}(\theta)\cos^{i-j}(\theta)\right)'\zeta_{m+1}(k\theta)d\theta.
    \end{align*}
    This implies that we obtained the same recursive relation in this case. Additionally, if $u\in C^{m+2}(\Omega)$, it is easy to see that equation \eqref{eq_rec_reln_m-1} holds even for $0\leq i\leq m+1$ as the step (\ref{int_by_parts_step}) is valid.
\end{proof}
\noindent
Now, using \eqref{Ck_int_by_parts_term}, $c_{k}$ in \eqref{Eq_Cheby_Coeff_intPart_one} can be re-expressed as $c_{k} = \frac{1}{2\pi} \frac{h}{k} I_{1}[\tilde{u}^{(1)},1,k]$. Performing integration by parts iteratively $M$ times provides the following lemma.
	\begin{lemma}\label{lemma_Cheby_Coeff_form}
		Let $u \in X^m(\Omega)$ and $1\leq M\leq m+1$, $k \in \mathbb{N}$, then
		\begin{equation}\label{eq_Cheby_Coeff_form}
			c_k = \sum_{i=1}^{M} \sum_{j=0}^{i} \frac{h^{i}}{k^M}\alpha_{ij}^{(M)}  I_{M}[\tilde{u}^{(i)},j,k],
		\end{equation}
		where $\alpha_{ij}^{(M)}$ are constants defined recursively and independent of $k$ and $h$, that is, for $0\le j \le i$, $1\le i \le M$, 
		\begin{gather*}
			\alpha_{ij}^{(M)} = \alpha_{1,ij}^{(M)}+\alpha_{2,ij}^{(M)}+\alpha_{3,ij}^{(M)},  \text{ where } \alpha_{1,ij}^{(M)}=\begin{cases}  
				0 , \,\, \text{if }i=M \text{ or } j=i\\
				(-1)^{M} (j+1)\alpha_{i(j+1)}^{(M-1)} , \text{ otherwise, }\end{cases}\\ \, \\                 
			\alpha_{2,ij}^{(M)}= \begin{cases}  
				0 , \,\,\text{if } i=M \text{ or } j=0 \\ 
				(-1)^{M} (i-j+1)\alpha_{i(j-1)}^{(M-1)}  , \text{ otherwise, }\end{cases} \alpha_{3,ij}^{(M)}=\begin{cases}  
				0 , \,\,\text{if } i=1 \text{ or } j=0 \\ 
				\frac{(-1)^{M-1}}{2}\alpha_{(i-1)(j-1)}^{M-1}  , \text{ otherwise, }\end{cases}
		\end{gather*}
	and $\alpha_{10}^{(1)}=0,\alpha_{11}^{(1)} = \frac{1}{2\pi}$. Additionally, $I_{m+1}[\tilde{u}^{(m+1)},j,k]=O\left(\frac{1}{k}\right)$, and if $u \in C^{m+2}(\Omega)$ then \eqref{eq_Cheby_Coeff_form} holds for $1 \le M \le m+2$.
\end{lemma}
\begin{proof}
We prove this lemma using induction on $M$. When $M=1$, the lemma is true since $c_{k} = \frac{1}{2\pi} \frac{h}{k} I_{1}[\tilde{u}^{(1)},1,k]$ as obtained from integration by parts. By induction hypothesis, assume that the result is true for any natural number $M\le m$. Using \Cref{lemma_expression_Ck_induction}, the coefficient $c_{k}$ can be expressed as
	\begin{align}
		c_k &=  \sum_{i=1}^{M} \sum_{j=0}^{i} \frac{h^{i}}{k^{M}}\alpha_{ij}^{(M)} \left[ \frac{(-1)^{M+1} j}{k}I_{M+1}[\tilde{u}^{(i)},j-1,k] + \frac{(-1)^{M} (i-j)}{k}I_{M+1}[\tilde{u}^{(i)},j+1,k] \right. \label{eq_rec_relation_m+2}\\ 
  &\quad + \left. \frac{(-1)^N}{2}\frac{h}{k}I_{M+1}[\tilde{u}^{(i+1)},j+1,k]\right]. \nonumber
	\end{align}
After re-indexing, $c_k$ becomes
    \begin{equation*}
         c_k= \sum_{i=1}^{M+1} \sum_{j=0}^{i} \frac{h^{i}}{k^{M+1}}\alpha_{ij}^{(M+1)} I_{M+1}[\tilde{u}^{(i)},j,k].
    \end{equation*}
 Thus, equation \eqref{eq_Cheby_Coeff_form} is true for $1\le M\le m+1$. This completes the proof by induction for $u\in X^{m}$. Additionally, since $\tilde{u}(\cos{(\cdot)})\in C^{m+2}_{\mathrm{pw}}[-\pi,\pi]$, there exist discontinuities of $(m+1)$-th derivative of $\tilde{u}(\cos{\theta})$ at $\{s_\ell\}_{\ell=0}^{n_d}$ for some $n_{d}\in \N$ such that $[-\pi,\pi] = \cup_{\ell=1}^{n_d} [s_{\ell-1},s_\ell]$. Therefore,
  \begin{align*} I_{m+1}[\tilde{u}^{(m+1)},j,k] &= \frac{(-1)^{m+1}}{k}\sum_{\ell=1}^{n_d}\left[ \left(\tilde{u}^{(m+1)}(\cos{\theta})\sin^{j}(\theta)\cos^{i-j}(\theta) \zeta_{m+2}(k\theta)\right)_{s_{\ell-1}}^{s_\ell} \right. \nonumber\\
     &\quad \left. -\int_{s_{\ell-1}}^{s_{\ell}} \left(\tilde{u}^{(m+1)}(\cos{\theta})\sin^{j}(\theta)\cos^{i-j}(\theta)\right)'\zeta_{m+2}(k\theta)d\theta \right] = O\left(\frac{1}{k}\right).
    \end{align*} 
    This completes the proof for case $u\in X^{m+2}(\Omega)$. If $u\in C^{m+2}(\Omega)$ then another application of recursive relation \Cref{lemma_expression_Ck_induction} proves \eqref{eq_Cheby_Coeff_form}  for $M=m+2$.
\end{proof}

While estimating the decay rate of $c_{k}$ in terms of $k$ and $h$, using \Cref{lemma_Cheby_Coeff_form} we have achieved the optimal decay in terms of $k$ and now, we would like to estimate the optimal bound in terms of $h$. To do so, we re-write the integral \eqref{Ck_int_by_parts_term} as a function of $\mu$ below and investigate it's behavior as $\mu \to 0$. 
\begin{equation}\label{Eq_expression_gijkl}
	I_{M}[\tilde{u}^{(i)},j,k](\mu) = \int_{-\pi}^{\pi} \tilde{u}^{(i)}\left(\frac{\mu}{h}\cos{\theta}\right)\sin^{j}(\theta)\cos^{i-j}(\theta)\zeta_{M}(k\theta)d\theta,
\end{equation}
where $\mu\in [-h, h]$. Observe that $I_{M}[\tilde{u}^{(i)},j,k](h) = I_{M}[\tilde{u}^{(i)},j,k]$. In the following lemma we estimate the bound for \eqref{Eq_expression_gijkl}. 
\begin{lemma}\label{lemma_Cheby_Coeff_Taylor_thm}
For integers $0 \leq j\leq i \leq m$ with $i<k$ and $0\leq M$ the following estimates hold:  
    \begin{enumerate}
    \item[i.] If $u \in X^{m}(\Omega)$ then 
	\begin{equation}\label{uXm}
		I_M[\tilde{u}^{(i)},j,k](\mu) = O\left(\mu^{\min\{k, m+1\}-i}\right).
	\end{equation}
    \item[ii.] If $u \in C^{m+2}(\Omega)$ then 
	\begin{equation}\label{uCm}
		I_M[\tilde{u}^{(i)},j,k](\mu) = O\left(\mu^{\min\{k, m+2\}-i}\right),  
    \end{equation}
    \end{enumerate}
    where $\mu \in [-h, h] \subseteq [-1,1]$.
\end{lemma}
\begin{proof}
(i). We first investigate the regularity of the integral $I_M[\tilde{u}^{(i)},j,k](\mu)$ defined in \eqref{Eq_expression_gijkl} at $\mu=0$. If $u\in X^{m}(\Omega)$, using Leibniz rule we deduce that $I_{M}[\tilde{u}^{(i)},j,k](\mu)$ is $m-i$ times differentiable. In particular the derivatives of $I_{M}[\tilde{u}^{(i)},j,k](\mu)$ are given by 
\begin{equation}\label{eq_derivative_of_integral}
	    \frac{d^{n}}{d\mu^{n}}I_{M}[\tilde{u}^{(i)},j,k](\mu) = \frac{1}{2^n}\int_{-\pi}^{\pi}\tilde{u}^{(i+n)}\left(\frac{\mu}{h}\cos{\theta}\right)\sin^{j}(\theta)\cos^{i-j+n}(\theta)\zeta_{M}(k\theta)d\theta = \frac{1}{2^n}I_{M}[\tilde{u}^{(i+n)},j,k](h),
	\end{equation}
	for $0 \le n \le m-i$. Since $u$ is piecewise $C^{m+2}$ which implies that $(m+1-i)^{th}$ derivative of $I_{M}[\tilde{u}^{(i)},j,k](h)$ exists. Now doing a finite Taylor expansion of $I_{M}[\tilde{u},j,k](\mu)$ at $\mu=0$, we have 
    \begin{equation}\label{eq_Taylor_expansion_Xm+2}
	I_{M}[\tilde{u}^{(i)},j,k](\mu) = \sum_{n = 0}^{m-i} I_{M}[\tilde{u}^{(i+n)},j,k](0) \frac{\mu^{n}}{2^n n!} + O\left(\mu^{m+1-i}\right).
	\end{equation}
    We now estimate the right hand side of equation \eqref{eq_Taylor_expansion_Xm+2}. For $0\le n \le m-i$, 
    \begin{equation}\label{eq_int_powsincos}
    I_{M}[\tilde{u}^{(i+n)},j,k](0) = \tilde{u}^{(i+n)}(0)\int_{-\pi}^{\pi} \sin^j{(\theta)}\cos^{i+n-j}{(\theta)}\zeta_M(k\theta) d\theta = 0,  
    \end{equation}
     if $0\le j \le i+n < k $ or $n<k-i$. Indeed, equation \eqref{eq_int_powsincos} follows by expanding powers of sine and cosine in terms of the complex exponential and using $\int_{-\pi}^{\pi}\cos{(
     k\theta)}d\theta = \int_{-\pi}^{\pi}\sin{(
     k\theta)}d\theta = 0$ for all non-zero $k$. Therefore, if $k \le m$, the first non-zero term of the summation in equation \eqref{eq_Taylor_expansion_Xm+2} occur at $n=k-i$ and we obtain $I_M[\tilde{u}^{i},j,k](\mu) = O(\mu^{k-i})$. If $k > m$, then every term in the summation of \eqref{eq_Taylor_expansion_Xm+2} becomes zero because of \eqref{eq_int_powsincos}, which implies $I_M[\tilde{u}^{i},j,k](\mu) = O(\mu^{m+1-i})$. Combining these two cases we get the desired result when $u\in X^{m}(\Omega)$. For part (ii), using Leibniz rule we deduce that $I_{M}[\tilde{u}^{(i)},j,k](\mu)$ is $m+2-i$ times differentiable. Now doing a finite Taylor expansion of $I_{M}[\tilde{u}^{(i)},j,k](\mu)$ at $\mu=0$, we have
    \begin{equation*}
	I_{M}[\tilde{u}^{(i)},j,k](\mu) = \sum_{n = 0}^{m+1-i} I_{M}[\tilde{u}^{(i+n)},j,k](0) \frac{\mu^{n}}{2^n n!} + O\left(\mu^{m+2-i}\right). 	
	\end{equation*}
    Now, following the similar arguments which we carried out for $u\in X^{m}(\Omega)$, we get the desired bound that $I_{M}[\tilde{u}^{(i)},j,k](\mu) = O(\mu^{k-i}) + O(\mu^{m+2-i})$ if $u\in C^{m+2}(\Omega)$.
\end{proof}
\begin{theorem}\label{Thm_Cheby_Coeff}
Let $k\geq 1$ be an integer and let $c_{k}$ denote the $k$-th Chebyshev coefficient of $\tilde{u}$ defined as in \eqref{eq_ul_def}. Then the following holds:
  \begin{enumerate}
    \item[i.] If $u \in X^{m}(\Omega)$ then 
	\begin{equation}
		|c_{k}| \le C\frac{h^{\min\{k,m+1\}}}{k^{m+2}}.
	\end{equation}
    \item[ii.] If $u \in C^{m+2}(\Omega)$ then 
	\begin{equation}
		|c_{k}| \le C\frac{h^{\min\{k,m+2\}}}{k^{m+2}},
    \end{equation}
    \end{enumerate}
	where $h = b - a$ is the length of $\Omega$ such that $h \le 1$ and $C$ is a positive constant independent of $k$ and $h$. 
\end{theorem}
\begin{proof}
	(i). If $u\in X^{m}(\Omega)$, using \Cref{lemma_Cheby_Coeff_form} for $M=m+1$, an expression for $c_{k}$ can be obtained as
	\begin{equation}\label{Eq_Cheby_Coeff_IBP_m+1}
		c_k = \sum_{i=1}^{m} \sum_{j=0}^{i} \alpha_{ij}^{(m+1)} \frac{h^{i}}{k^{m+1}} I_{m+1}[\tilde{u}^{(i)},j,k]+\sum_{j=0}^{m+1} \alpha_{(m+1)j}^{(m+1)} \frac{h^{m+1}}{k^{m+1}} I_{m+1}[\tilde{u}^{(m+1)},j,k].
	\end{equation}
Additionally, in \Cref{lemma_Cheby_Coeff_form} we have shown that $I_{m+1}[\tilde{u}^{(m+1)},j,k] =O\left(1/k\right)$. Using this bound, the second term in \eqref{Eq_Cheby_Coeff_IBP_m+1} can be estimated as 
	\begin{equation*}
		\left|\sum_{j=0}^{m+1} \alpha_{(m+1)j}^{(m+1)} \frac{h^{m+1}}{k^{m+1}} I_{m+1}[\tilde{u}^{(m+1)},j,k]\right| \le C\frac{h^{m+1}}{k^{m+2}}.
	\end{equation*}
Using the recurrence relation for $I_{m+1}[\tilde{u}^{(i)},j,k]$ described in \Cref{lemma_expression_Ck_induction}, the first sum in equation \eqref{Eq_Cheby_Coeff_IBP_m+1} can be expressed as
	\begin{align*} 
		c_k &= \sum_{i=1}^{m} \frac{h^{i}}{k^{m+2}}\sum_{j=0}^{i} \alpha_{ij}^{(m+1)} \left[ (-1)^{m}j I_{m+2}[\tilde{u}^{(i)},j-1,k]-(-1)^m (i-j)  I_{m+2}[\tilde{u}^{(i)},j+1,k] \right. \\
        &\quad  \left. - \frac{(-1)^m}{2}h I_{m+2}[\tilde{u}^{(i+1)},j+1,k] \right].
	\end{align*}
	Thus, using \Cref{lemma_Cheby_Coeff_Taylor_thm} and observing the fact $I_{M}[\tilde{u}^{(i)},j,k](h) = I_{M}[\tilde{u}^{(i)},j,k]$, we obtain 
    \begin{align*}
        \left|I_{m+2}[\tilde{u}^{(i)},j\pm 1,k]\right| &\le  C \, h^{\min\{k,m+1\}-i}, \\
        \left|I_{m+2}[\tilde{u}^{(i+1)},j+1,k]\right| &\le C \, h^{\min\{k,m+1\}-(i+1)}.
    \end{align*}
Since the constants in the above summations are dependent only on $i$, $j$ and $m$, and independent of $h$ and $k$ we obtain the required result. For part (ii), if $u\in C^{m+2}(\Omega)$, using \Cref{lemma_Cheby_Coeff_form} for $M=m+2$, an expression for $c_{k}$ can be obtained as
	\begin{align}\label{eq_Cheby_coeff_uCm+2}
	    c_k &= \sum_{i=1}^{m+2} \sum_{j=0}^{i} \frac{h^{i}}{k^{m+2}}\alpha_{ij}^{(m+2)}  I_{m+2}[\tilde{u}^{(i)},j,k], 
	\end{align}
		using \Cref{lemma_Cheby_Coeff_Taylor_thm} and the fact $I_{m+2}[\tilde{u}^{(i)},j,k] = I_{m+2}[\tilde{u}^{(i)},j,k](h)$, we get 
        \begin{equation}
        \left|I_{m+2}[\tilde{u}^{(i)},j,k]\right| \le C \, h^{\min\{k,m+2\}-i}.
        \end{equation}
        Since $\alpha_{ij}^{m+2}$ in equation \eqref{eq_Cheby_coeff_uCm+2} depends only on $i$, $j$ and $m$, and independent of $h$ and $k$, we obtain the required bound for $c_{k}$ when $u\in C^{m+2}(\Omega)$.        
\end{proof}
In view of the singular integration scheme, note that it is important to investigate the convergence of the discrete Chebyshev expansion of a function $u\in X^{m}(\Omega)$ and the error in $|c_{k}-\tilde{c}_{k}|$ which is the approximation of the continuous coefficients $c_{k}$ by their discrete counterparts $\tilde{c}_{k}$. The following result briefly describes the error in the approximation $|c_{k}-\tilde{c}_{k}|$.
\begin{theorem}
\label{lemma_error_discrete_and_continuous_ChebyCoeff}
Let $0\leq k\leq n-1$ be an integer and $u_{l}$ as defined in \eqref{eq_ul_def}. Let $c_{k}$ denote the $k$-th Chebyshev coefficient of $\tilde{u}$ and $\tilde{c}_{k}$ its discrete counterpart, then the following hold: 
\begin{enumerate}
    \item[i.] If $u \in X^{m}(\Omega)$ then  \begin{equation}
    \left|c_{k}-\tilde{c}_{k}\right|\le C \frac{h^{\min\{n+1,m+1\}}}{n^{m+2}}.
    \end{equation}
    \item[ii.] If $u \in C^{m+2}(\Omega)$ then   
    \begin{equation}
    \left|c_{k}-\tilde{c}_{k}\right|\le C \frac{h^{\min\{n+1,m+2\}}}{n^{m+2}},
    \end{equation}
\end{enumerate}
where $h \le 1$ is the length of $\Omega$, $n$ is the number of terms in discrete Chebyshev expansion of $\tilde{u}$, that is, $\tilde{u}(t) \approx \sum_{k=0}^{n-1} \tilde{c}_{k} T_k(t)$, and $C$ is some constant independent of $h$ and $n$. 
\end{theorem}
\begin{proof}
	The Chebyshev series expansion of $u_{l}$ is given by
    \begin{equation}\label{Cheby_series_error}
		\tilde{u}(t)=\sum\limits_{k=0}^{\infty} c_{k}T_{k}(t) = \sum\limits_{k=0}^{n-1} c_{k} T_{k}(t)+R_{n}(t), \text{ where }R_{n}(t) = \sum\limits_{k=n}^{\infty} c_{k} T_{k}(t).
	\end{equation}
	From equation (\ref{Cheby_series_error}) it is clear that $(\tilde{u}-R_{n})(t)$ is a polynomial with degree at most $n-1$ implying that it's discrete and continuous Chebyshev coefficients are equal \cite[Theorem 6.7]{mason2002chebyshev}. We obtain 
	\begin{equation}
		\left(\tilde{u}-R_{n}\right)(t)=\sum_{k=0}^{n-1} c_{k}T_{k}(t), \text{ and }
		c_{k}=\frac{\gamma_{k}}{n} \sum_{i=0}^{n-1}\left(\tilde{u}-R_{n}\right)\left(t_{i}\right) T_{k}\left(t_{i}\right),\ \gamma_{k} = 
		\left\{
		\begin{array}{ll}
			1, & \mbox{if } k=0 \\
			2, & \mbox{otherwise},
		\end{array}
		\right.
	\end{equation}
    where $t_i = \cos{\left(\pi \frac{2i+1}{2n}\right)}, \ 0 \le i < n$.
	The error in the continuous and the discrete Chebyshev coefficients becomes
	\begin{equation}\label{Discrete_orth_cheby}
		\tilde{c}_{k}-c_{k} =\frac{\gamma_{k}}{n} \sum_{i=0}^{n-1} R_{n}\left(t_{i}\right) T_{k}\left(t_{i}\right) = \frac{\gamma_{k}}{n} \sum\limits_{j=n}^{\infty} c_{j}\left[\sum\limits_{i=0}^{n-1} T_{j}\left(t_{i}\right) T_{k}\left(t_{i}\right)\right].    
	\end{equation}
Using \cite[Section 4.6]{mason2002chebyshev}, one can derive the following version of the discrete orthogonality rule: For integers $0\leq k \leq n-1$ and $1 \le n\le j$ is given by 
\begin{equation}\label{eqn_discrete_erthogonality}
	\sum\limits_{i=0}^{n-1} T_{j}\left(y_{i}\right)T_{k}\left(y_{i}\right)=\frac{1}{2}\left\{\begin{array}{cl}
			(-1)^{l} n, & \text {if }~ j+k=2 n l, ~l \in \N \\
			(-1)^{l} n, & \text {if }~ j-k=2 n l, ~l \in \N \\
			0, & \text {otherwise, }
		\end{array}\right.
	   \end{equation} 
       where $y_i = \cos{\left(\pi \frac{2i+1}{2n}\right)}, \ 0 \le i < n$.  Using the identity \eqref{eqn_discrete_erthogonality}, we obtain
	\begin{equation}
    \tilde{c}_{k}-c_{k} = \frac{\gamma_{k}}{2}\left[\sum\limits_{i=1}^{\infty} (-1)^{i} c_{2 n i-k}+\sum\limits_{i=1}^{\infty}(-1)^{i} c_{2ni+k}\right].
	\end{equation}
	For every $0\leq k < n$, using \Cref{Thm_Cheby_Coeff}, we deduce
    \begin{equation*}
        \left|c_{k}-\tilde{c}_{k}\right| \leq c \sum_{i=1}^{\infty}\left[\frac{h^{\min\{2ni-k,m+1\}}}{(2 n i-k)^{m+2}}+\frac{h^{\min\{2ni+k,m+1\}}}{(2 n i+k)^{m+2}}\right].
    \end{equation*}
    To capture the appropriate power of $h$, we split the summation in what follows. There exist integer $N_{0} \ge 1$ such that $2n(N_{0}-1)+n-1\leq m+1<2nN_{0}$. For this choice of $N_0$, observe that the exponent on $h$ has the following three cases.
    \begin{enumerate}
        \item [i.] If $1\leq i\leq N_{0}-1$ then $2ni-k<2ni+k\leq m+1$ for each $0 \le k \le n-1$.
        \item [ii.] If $i= N_{0}$ then $n+1\le 2nN_0-k $ and $m+1< 2nN_{0}+k$ for each $0 \le k \le n-1$.
        \item [iii.] If $i\geq N_{0}+1$ then $m+1\leq 2ni-k<2ni+k$ for each $0 \le k \le n-1$.
    \end{enumerate}
    Therefore, we get 
\begin{align*}
       \left|c_{k}-\tilde{c}_{k}\right| &\leq c \frac{h^{n+1}}{n^{m+2}} \left(\sum_{i=1}^{N_{0}-1}\left[\frac{h^{n(2i-1)-k-1}}{(2i-\frac{k}{n})^{m+2}}+\frac{h^{n(2i-1)+k-1}}{(2i+\frac{k}{n})^{m+2}}\right]\right)  + \frac{h^{\min\{2nN_{0}-k,m+1\}}}{(2N_{0}-\frac{k}{n})^{m+2}} \\
       &\quad + c\frac{h^{m+1}}{n^{m+2}}\left(\frac{1}{(2N_{0}+\frac{k}{n})^{m+2}}+ \sum_{i=N_{0}+1}^{\infty}\left[\frac{1}{(2i-\frac{k}{n})^{m+2}}+\frac{1}{(2i+\frac{k}{n})^{m+2}}\right] \right).
\end{align*}
	For $0 \leq k \leq n-1$, the terms $\frac{1}{\left(2 i-\frac{k}{n}\right)^{m+2}}$ and $\frac{1}{\left(2 i+\frac{k}{n}\right)^{m+2}}$ are bounded by $\frac{1}{(2 i-1)^{m+2}}$ and $\frac{1}{(2 i)^{m+2}}$ respectively. Thus, the infinite series is finite and the result follows. The proof when $u\in C^{m+2}$ can be carried out in a similar manner.
\end{proof}
\begin{remark}\label{rem_disCheby_Coeff}
    Let $n$ be a positive integer and $\tilde{c}_{k}$ denote the $k$-th discrete Chebyshev coefficient of $\tilde{u}$ defined as \eqref{eq_ul_def} for $0 \le k <n$. Using triangle inequalities in \Cref{lemma_error_discrete_and_continuous_ChebyCoeff} and  \Cref{Thm_Cheby_Coeff}, the following holds:
  \begin{enumerate}
    \item[i.] If $u \in X^{m}(\Omega)$ then 
	\begin{equation}
		\left|\tilde{c}_{k}\right| \le C \frac{h^{\min\{k,m+1\}}}{k^{m+2}}.   
	\end{equation}
    \item[ii.] If $u \in C^{m+2}(\Omega)$ then 
	\begin{equation}
		\left|\tilde{c}_{k}\right| \leq C \frac{h^{\min\{k,m+2\}}}{k^{m+2}},
    \end{equation}
    \end{enumerate}
	where $h = b-a$ is the length of $\Omega$ and $C$ is some constant independent of $h$ and $k$. 
\end{remark}
Since $T_{k}(t)$ is a polynomial, the approximation $\sum_{k=0}^{n-1}\tilde{c}_k T_k(t)$ is a polynomial approximation to $\tilde{u}$. More precisely, denoting $L^{n}[\tilde{u}](t) = \sum_{k=0}^{n-1}\tilde{c}_{k}T_{k}(t)$, it is easy to see that $L^{n}[\tilde{u}]$ is the $n$-th order Lagrange interpolating polynomial corresponding to the roots of the Chebyshev polynomial $T_{n}(t)$. Interestingly, using \cite[Section 2.5.5]{davis2014methods}, one can notice that the FF-Rule is an interpolatory quadrature and 	
\begin{equation}\label{eq_FFQ_interpolatory}
  \I_{ff}^{n}[\tilde{u}] = \frac{h}{2}\int\limits_{-1}^{1}L^{n}[\tilde{u}](t)dt.  
\end{equation}
In the following theorem we describe a convergence result for the FF-Rule approximation of the integral of $\tilde{u}$, which follows from the results described above.
\begin{theorem}\label{theorem_quad_convergence}
Let $n$ be a positive integer and $\I_{ff}^{n}[\tilde{u}]$ be the FF-Rule approximation of the integral of $u$ on interval $\Omega$ with $n$ discretization points where $\tilde{u}$ is defined as in \eqref{eq_ul_def}. Then the following holds: 
  \begin{enumerate}
    \item[i.] If $u \in X^{m}(\Omega)$ then
    \begin{equation*}
    \left|\int_{a}^{b} u(y)dy - \I_{ff}^{n}[\tilde{u}] \right| \le C \frac{h^{\min\{n+1+n_0,m+2\}}}{n^{m+2}}.
    \end{equation*}
    \item[ii.] If $u \in C^{m+2}(\Omega)$ then
\begin{equation*}
    \left|\int_{a}^{b} u(y)dy - \I_{ff}^{n}[\tilde{u}] \right| \le C \frac{h^{\min\{n+1+n_0,m+3\}}}{n^{m+2}}, 
    \end{equation*}
\end{enumerate}
    where $n_0 = 1$ if $n$ is odd and $n_0=0$ otherwise, $h \le 1$ and $C$ is a constant independent of $h$ and $n$. 
\end{theorem}
\begin{proof}
    Using equation  \eqref{eq_FFQ_interpolatory} and the identity related to the integral of a Chebyshev polynomial given in \cite[Section 2.5.5]{davis2014methods}, one can notice that the error in the FF-Rule approximation can be written as follows:
	\begin{align}\label{eq_subst_3.6}
		\left|\int_{a}^{b}u(y)~dy - \I_{ff}^{n}[\tilde{u}] \right| &= \left|\frac{h}{2}\int_{-1}^{1}\tilde{u}(t)~dt - \frac{h}{2}\int_{-1}^{1}L^{n}[\tilde{u}](t)dt\right| \nonumber
        \\
        &= \frac{h}{2}\left|\int_{-1}^{1}\sum\limits_{k=0}^{\infty}c_{k}T_{k}(t)~dt - \int_{-1}^{1}\sum\limits_{k=0}^{n-1}\tilde{c}_{k}T_{k}(t)~dt\right| \nonumber\\
		&\leq \frac{h}{2}\sum\limits_{k=n+n_0}^{\infty}|c_{k}| \left|\int_{-1}^{1}T_{k}(t)~dt \right| + \frac{h}{2}\sum\limits_{k=0}^{n-1+n_0}|c_{k} - \tilde{c}_{k}|\left| \int_{-1}^{1}T_{k}(t)~dt\right|.
	\end{align}
	Now using \Cref{lemma_error_discrete_and_continuous_ChebyCoeff} for $u\in X^{m}$ and $u \in C^{m+2}$ in equation \eqref{eq_subst_3.6} we can prove (i) and (ii) respectively.
\end{proof} 
The decay rate of discrete Chebyshev coefficients presented in \Cref{rem_disCheby_Coeff} and the FF-rule approximation derived in \Cref{theorem_quad_convergence} are demonstrated numerically the following examples. 
\begin{example}
Consider the function $u(x) = x^3|x|+x^3+x^2+x+1$ in the interval $\left[-\frac{1}{2N},\frac{1}{N}\right]$ where $N = 2,4,8,\cdots,256$. That is, the length of the interval $h=\frac{3}{2N}$. Table (\ref{tab:convcoeff_exp}) demonstrates the decay rates of the discrete Chebyshev coefficients $\tilde{c}_{k}$ of $\tilde{u}(x) = u\left(\xi(x)\right)$, where $1\leq k\leq 10$, $\xi(t)= ht/2 +1/4N$. Since $u\in X^{m}[-1, 1]$, with $m=3$, according to \Cref{rem_disCheby_Coeff}, $\tilde{c}_{k}$ converge to zero with the rate $\min\{k,4\}$ as $h\rightarrow 0$, for every $k\geq 1$, which is what the columns in the table (\ref{tab:convcoeff_exp}) show. Hence, the numerical convergence rates exactly align with the theoretically estimated convergence rates.  

\begin{table}[H]
	\centering
	\begin{tabular}{|c|c|c|c|c|c|c|c|c|c|c|} \hline 
		$N$&  $\tilde{c}_1$&  $\tilde{c}_2$&  $\tilde{c}_3$ &$\tilde{c}_4$&  $\tilde{c}_5$&  $\tilde{c}_6$& $\tilde{c}_7$ & $\tilde{c}_8$& $\tilde{c}_9$ & $\tilde{c}_{10}$\\ \hline 
		4&  $1.33$&  $2.35$&   $3.30$&$4.00$&  $4.00$&  $4.00$& $4.00$ & $4.00$ & $4.00$ & $4.00$\\ \hline 
		8&  $1.13$&  $2.16$&   $3.18$&$4.00$&  $4.00$&  $4.00$& $4.00$ & $4.00$ & $4.00$ & $4.00$\\ \hline 
		16&  $1.05$&  $2.08$&   $3.10$&$4.00$&  $4.00$&  $4.00$& $4.00$ & $4.00$ & $4.00$ & $4.00$\\ \hline 
		32&  $1.02$&  $2.04$&   $3.05$&$4.00$&  $4.00$&  $4.00$& $4.00$ & $4.00$ & $4.00$ & $4.00$\\ \hline 
		64&  $1.01$&  $2.02$&   $3.03$&$4.00$&  $4.00$&  $4.00$& $4.00$ & $4.00$ & $4.00$ & $4.00$\\ \hline 
		128&  $1.01$&  $2.01$&   $3.01$&$4.00$&  $4.00$&  $4.00$& $4.00$ & $4.00$ & $4.00$ & $4.00$\\ \hline 
		256&  $1.00$&  $2.00$&   $3.01$&$4.00$&  $4.00$&  $4.00$& $4.00$ & $4.00$ & $4.01$ & $4.00$\\ \hline
	\end{tabular}
	\caption{Order of convergence of Chebyshev coefficients $\tilde{c}_{k}$ of $\tilde{u}\in X^{m}[-1, 1]$ with $m=3$. The expected decay rate of $\tilde{c}_{k}$ is $\min\{k,m+1\}=\min\{k,4\}$. The theoretical decay rate match with their numerical counterparts.}
	\label{tab:convcoeff_exp}
\end{table}
\end{example}

\section{Quadrature Analysis}\label{error}
In this section, we present error analysis for single patch and later generalize this idea for arbitrary number of patches. We consider the integration patch $\Omega=[a,b]$ and length $h = b-a$. For analysis in single patch we avoid mentioning patch number $\ell$. Recall from \Cref{method}, for the target point $x\in\Omega$,
\begin{equation}
		\I[u](x) = \int\limits_{-1}^{1} g_{\alpha}\left(|x-\xi(t) | \right) u(\xi(t)) dt, \
        \beta_{k}(x) = \frac{h}{2}\int_{-1}^{1} g_{\alpha}(|x-\xi(t)|)T_{k}(t)dt,
\end{equation}
where $\xi(t) = \frac{h}{2}t+\frac{a+b}{2}$ and the kernel is defined as in \eqref{eq_kernel}. The error in singular quadrature is given by 
\begin{equation}\label{Eq_total_err_sing_int}
	E_{\textrm{sing}}^{n}[u](x) = \left|\I[u](x)-\I_{\textrm{sing}}^{n}[u](x)\right| \leq   E^{n}_{1}[u](x)+E^{n}_{2}[u](x)+E^{n}_{3}[u](x), 
\end{equation}
where $\I_{\textrm{sing}}^{n}[u]$ is defined as in equation \eqref{sing_quad}, and $E^{n}_{1}[u],E^{n}_{2}[u],E^{n}_{3}[u]$ are given by
\begin{align}
	E^{n}_{1}[u](x) &= \left|\I[u](x)- \sum\limits_{k=0}^{n-1} c_{k}\beta_{k}(x)\right|, \label{eq_error1}\\
	E^{n}_{2}[u](x) &=  \sum\limits_{k=0}^{n-1}\left|c_{k}-\tilde{c}_{k}\right| \left|\beta_{k}(x)\right|,\label{eq_error2} \\ 
	E^{n}_{3}[u](x) &= \sum\limits_{k=0}^{n-1} \left|\tilde{c}_{k} \right| \left|\beta_{k}-\tilde{\beta}_{k}(x)\right|. \label{eq_error3}
\end{align}
 Recall from \Cref{method}, $\tilde{\beta}_{k}$ is defined as in equation \eqref{sing_quad}, and $c_{k},\tilde{c}_{k}$ are the $k$-th continuous and discrete Chebyshev coefficient of the density $u$ respectively. Term $E^{n}_{1}$ denotes error in truncation of Chebyshev expansion, $E^{n}_{2}$ is the error in the approximation of  continuous Chebyshev coefficients $c_{k}$ by discrete coefficients $\tilde{c}_{k}$, and $E^{n}_{3}$ represents error the approximation of the weights $\beta_{k}(x)$ by the singular integration strategy discussed in \Cref{method}. In the following we establish an error estimate for all the three error terms separately. 

\subsection{Error in truncation}\label{subsection_error_trunc}
We now compute the bound for the first error term $E_{1}^{n}[u]$, for that we need the decay rate of $c_{k}$ on $\Omega$, which we derived in \Cref{Thm_Cheby_Coeff}, and in what follows, we find a bound for the continuous wights $\beta_{k}(x)$, which holds uniformly for each $x \in \Omega$. 
\begin{lemma}\label{lemma_bound_on_beta_k}
	For $x \in \Omega$,
	\begin{equation} \left|\beta_{k}(x)\right| \leq C\begin{dcases}
		\frac{h |\log{h}|}{k^2}+\frac{h\log{k}}{k}, & \text{ if } \alpha = 0, \ k>1 \\ 
		\frac{h^{1-\alpha}}{k^{1-\alpha}}, & \text{ if } 0<\alpha<1, \ k\ge 1.
	\end{dcases}
    \end{equation}
\end{lemma}
\begin{proof}
    For $x \in \Omega$, there exists $t_x \in [-1,1]$ such that $\xi(t_x) = x$. Therefore, $\beta_k$ can be written as \begin{equation}\label{beta_int_cases}
        \beta_k(x) = \frac{h}{2}\int_{-1}^{1} g_{\alpha}(|\xi(t_x)-\xi(t)|) T_k(t) dt = \begin{dcases}
            \frac{h}{2}\int_{-1}^1 \log{\left(\left|\frac{h}{2}(t_x-t)\right|\right)} T_k(t)dt, \text{ if } \alpha =0\\
            \left(\frac{h}{2}\right)^{1-\alpha}\int_{-1}^1 \left|t_x-t\right|^{-\alpha} T_k(t)dt, \text{ if } 0<\alpha<1.
        \end{dcases}
    \end{equation}
    For $\alpha=0$, the integral can split as \begin{equation}\label{beta_log_kernel_split}
        \beta_k(x) = \frac{h}{2}\log{\left(\frac{h}{2}\right)}\int_{-1}^1 T_k(t)dt + \frac{h}{2}\int_{-1}^1 \log{\left|t_x-t\right|} T_k(t)dt.
    \end{equation}
    The integral of $k$-th Chebyshev polynomial is $\int_{-1}^1 T_k(t)dt = O\left(1/k^2\right)$ \cite{xiang2013convergence}. Now utilizing \cite[Lemma 3.1]{dominguez2013filon} for case $0<\alpha<1$ in equation \eqref{beta_int_cases} and log kernel in \eqref{beta_log_kernel_split}, we obtain the required result. 
\end{proof}
\noindent
Now, we present the theorem which gives a bound for $E^{n}_{1}$ in equation \eqref{eq_error1}.
\begin{theorem}\label{theorem_singular_1d_error_first_term}
	If $u \in X^{m}(\Omega)$ then
	\[  E_{1}^{n}[u](x) \le  C\displaystyle\begin{dcases}
		\frac{h^{\min\{n+1,m+2\}}|\log{h}|}{n^{m+3}}+ \frac{h^{\min\{n+1,m+2\}} \log{n}}{n^{m+2}}, & \text{ if } \alpha = 0 \\ 
		\frac{h^{\min\{n+1-\alpha,m+2-\alpha\}}}{n^{m+2-\alpha}}, & \text{ if } 0<\alpha<1.
	\end{dcases}\]
    where $h$ is length of $\Omega$, $n$ is number of terms in Chebyshev expansion of $u$ and $C$ is some constant independent of $h$ and $n$.
\end{theorem}

\begin{proof}
	The Chebyshev series of $u$ converges absolutely and uniformly. Therefore, employing the Chebyshev series expansion of $u(y)$, we have
	\begin{equation}
    E_{1}^{n}[u](x) = \left|\I[u](x) - \sum\limits_{k=0}^{n-1} c_{k}\beta_{k}(x)\right| = \left|\int\limits_{-1}^{1} g_{\alpha}(|x-\xi(t)|) \left[\sum\limits_{k=n}^{\infty} c_{k} T_{k}(t) \right] d t \right| \leq \sum\limits_{k=n}^{\infty}\left|c_{k}\right|\left|\beta_{k}(x)\right|.
	\end{equation}
	Employing \Cref{Thm_Cheby_Coeff} and \Cref{lemma_bound_on_beta_k}, we have
	\begin{align*}
        E_{1}^{n}[u](x) &\leq C h^{\min\{n+1,m+2\}}\left[\sum_{k=n}^{\infty} \frac{\log{k}}{k^{m+3}}+\log(h)\sum_{k=n}^{\infty} \frac{1}{k^{m+4}}\right]\leq C h^{\min\{n+1-\alpha,m+2-\alpha\}}\sum_{k=n}^{\infty} \frac{1}{k^{m+3-\alpha}},
	\end{align*}
	for $\alpha=0$ and $0<\alpha<1$ respectively. Hence, from Euler's summation formula \cite[Theorem 3.1]{apostol2013introduction}, the theorem follows.
\end{proof}	
We proceed to estimate the error term $E_{2}^{n}$ in \eqref{eq_error2}, which represents the approximation of continuous Chebyshev coefficients by their discrete counterparts. 
\begin{theorem}\label{theorem_singular_1d_error_second_term}
	If $u \in X^{m}(\Omega)$  then
	\[ E_{2}^{n}[u](x)\le C \begin{dcases}
	\frac{h^{\min\{n+2,m+2\}}(\log{n})^{2}}{n^{m+2}}+\frac{h^{\min\{n+2,m+2\}} |\log{h}|}{n^{m+3}}, & \text{ if } \alpha=0\\
	\frac{h^{\min\{n+2-\alpha,m+2-\alpha\}}}{n^{m+2-\alpha}}, & \text{ if } 0<\alpha<1,
	\end{dcases}\]
    where $h$ is length of $\Omega$, $n$ is number of terms in Chebyshev expansion of $u$ and $C$ is some constant independent of $h$ and $n$.
\end{theorem}

\begin{proof}
	For $0 < \alpha < 1$, it is easy to see that $|\beta_{0}(x)| \leq Ch^{1-\alpha}$ and $\sum \limits_{k=1}^{n-1}\frac{1}{k^{1-\alpha}} = O(n^{\alpha})$ using Euler summation formula \cite[Theorem 3.1]{apostol2013introduction}. Utilizing Lemma (\ref{lemma_bound_on_beta_k}) for estimating $|\beta_{k}(x)|$, we obtain
	\begin{equation}\label{eqn_finite_sum_beta_non_log}
		\sum_{k=0}^{n-1}|\beta_{k}(x)| \leq |\beta_{0}(x)| +  C\sum_{k=1}^{n-1}\frac{h^{1-\alpha}}{k^{1-\alpha}} = O(h^{1-\alpha} n^{\alpha}).
	\end{equation}
 Similarly, for $\alpha=0$, $|\beta_{0}(x)| \leq C h |\log{h}|$, $|\beta_{1}(x)| \leq Ch|\log{h}|$ and $\sum \limits_{k=1}^{n-1}\frac{\log{k}}{k} = O((\log{n})^{2})$ using Euler summation formula \cite[Theorem 3.1]{apostol2013introduction}. We obtain
\begin{align}
		\sum\limits_{k=0}^{n-1}|\beta_{k}(x)| &\leq |\beta_{0}(x)| + |\beta_{1}(x)| + C h\sum\limits_{k=2}^{n-1}\frac{\log{k}}{k}+C h|\log{h}| \sum\limits_{k=2}^{n-1}\frac{1}{k^2} \le C \left[ h(\log{n})^{2}+\frac{h |\log{h}|}{n}\right]. \label{eqn_finite_sum_beta_log}
	\end{align}
	By applying \Cref{lemma_error_discrete_and_continuous_ChebyCoeff}, the error term $E_{2}^{n}$ can be deduced to
	\begin{equation*}
		E_{2}^{n}[u](x) = \sum\limits_{k=0}^{n-1}\left|c_{k}-\tilde{c}_{k}\right| \left|\beta_{k}(x)\right| 
		\leq C\frac{h^{\min\{n+1,m+1\}}}{n^{m+2}} \sum\limits_{k=0}^{n-1} \left|\beta_{k}(x)\right|.
	\end{equation*}
	Thus, the proof follows from equations (\ref{eqn_finite_sum_beta_non_log}) and (\ref{eqn_finite_sum_beta_log}), respectively.
\end{proof}

\subsection{Error in the approximation of weights}\label{subsection_error_weight}
Before finding a bound on error term $E_3^n$ \eqref{eq_error3}, let us recall the change of variable which we incorporated in singular integral, that is, for $p\geq 2$
\begin{equation*}
		\psi_{p}(t) = \frac{2\left[v_{p}(t)\right]^{p}}{\left[v_{p}(t)\right]^{p} + \left[v_{p}(- t)\right]^{p}}, \ 
		v_{p}(t) =
		\left(\frac{1}{2} - \frac{1}{p} \right) t^{3} + \frac{t}{p} + \frac{1}{2}, ~-1 \leq t \leq 1.
\end{equation*}
Now, we state properties of the change of variables $\psi_p$ and $v_p$ in the following lemma.
\begin{lemma}\label{lemma_wp_properties}
	For an integer $p\geq 2$, the following holds:
	\begin{enumerate}
		\item  The function $v_{p}$ is strictly increasing and $v_{p}(t)=(t+1) \cdot q_{p}(t)$, where $q_{p}(t) = \left(\frac{1}{2}-\frac{1}{p}\right) (t^{2}-t)+\frac{1}{2}$.
		\item For $t\in [-1,1]$, $q_{p}(t)>0$ and $\psi_{p}(t) = (t+1)^{p} \cdot Q_{p}(t)$, where $Q_{p}(t) = \frac{2 \left[q_{p}(t)\right]^{p}}{[v_{p}(t)]^{p}+[v_{p}(-t)]^{p}}$.
		
		\item The derivative of $\psi_p$ is 
		$\psi_{p}^{\prime}(t) = (t+1)^{p-1}\cdot R_{p}(t),$ where $R_{p}(t) = (t+1)\cdot Q_{p}^{\prime}(t) + p\cdot Q_{p}(t).$
		\item  The function $\psi_{p}$ has a zero at $t=-1$ of order strictly $p$.
		\item The functions $Q_{p}$, $1/Q_{p}$ and $R_{p}$ are infinitely differentiable on $[-1, 1]$.
	\end{enumerate}
\end{lemma}
\begin{proof}
	Assertions $(1)$, $(2)$ and $(3)$ follow from the definition of $v_{p}$ and $\psi_{p}$. Clearly $v_{p}(-1) = 0$, $\psi_{p}(-1) = 0$. Using the Leibniz rule of derivative,  we have 
	\begin{equation*}
		\psi_{p}^{(j)}(t) = \sum\limits_{i=0}^{j-1} \dbinom{j}{i} \left[ \frac{d^{i}}{dt^{i}} (t+1)^{p} \right]  Q_{p}^{(j-i)}(t) + \frac{p!}{(p-j)!} (t+1)^{p-j} Q_{p}(t).
	\end{equation*}
	Therefore, for $1\leq j \leq p-1$, $\psi_{p}^{(j)}(-1) = 0$. Since $Q_{p}(t) \neq 0$ for every $t\in[-1,1]$, $\psi_{p}^{(p)}(-1) \neq 0$. Hence (4) is proved.
	Since the numerator and denominator of $Q_{p}$ are the non-zero polynomials in $t$, $Q_{p}$ and $1/Q_{p}$ are infinitely differentiable. Moreover, it is clear that the smoothness of $Q_{p}$ implies the smoothness of $R_{p}$. Thus, the assertion $(5)$ is proved.	
\end{proof}
\noindent
To find a bound on $E_3^n$ \eqref{eq_error3} which is,
\begin{equation*}
		E_{3}^{n}[u](x) =\left|\sum\limits_{k=0}^{n-1} \tilde{c}_{k}\left(\beta_{k}(x)-\tilde{\beta}_{k}(x)\right)\right| 
		\leq \sum\limits_{k=0}^{n-1}\left|\tilde{c}_{k}\right| \left|\beta_{k}(x)-\tilde{\beta}_{k}(x)\right|,
\end{equation*}
it is important to compute bounds for $|\beta_{k}(x) - \tilde{\beta}_{k}(x)|$ uniformly in $x$ and $k$. Therefore, consider 
\begin{equation}\label{eq_err_sum_beta_kRandL}
 |\beta_{k}(x) - \tilde{\beta}_{k}(x)| \leq |\beta_{k,L}(x) - \tilde{\beta}_{k,L}(x)| + |\beta_{k,R}(x) - \tilde{\beta}_{k,R}(x)|,   
\end{equation}
where $\beta_{k,L}$, $\beta_{k,R}$, $\tilde{\beta}_{k,L}$ and $\tilde{\beta}_{k,R}$ are as defined in \eqref{eq_weights_betaLR} and \eqref{sing_quad} respectively for the single patch $\Omega$. Now, we find the estimates for $|\beta_{k}(x) - \tilde{\beta}_{k}(x)|$ by combining the errors for both terms on the right hand-side of \eqref{eq_err_sum_beta_kRandL} for the cases $\alpha=0$ and $0<\alpha<1$ separately. We start with $|\beta_{k,L}(x) - \tilde{\beta}_{k,L}(x)|$. For every $x \in \Omega$, define $\xi(t_x) = x$, where $\xi : [-1,1] \to \Omega$ is defined by $\xi(t) = \frac{h}{2}t+\frac{a+b}{2}$. The expression for $\beta_{k,L}$ and $\beta_{k,R}$ for $x \in \Omega \setminus \{a,b\}$ can be simplified as
\begin{align}
	\beta_{k,L}(x) &=  \frac{h}{2}\int_{-1}^{1} g_{\alpha}\left(\frac{h}{2}(t_x+1)\psi_{p}\left(-\frac{1+\tau}{2}\right)\right)T_{k}(\psi_{p,L}(t_x,\tau))\psi^{\prime}_{p,L}(t_x,\tau) d\tau, \label{eq_betaL_COV} \\
    \beta_{k,R}(x) &= \frac{h}{2}\int_{-1}^{1} g_{\alpha}\left(\frac{h}{2}(1-t_x)\psi_{p}\left(-\frac{1-\tau}{2}\right)\right)T_{k}(\psi_{p,R}(t_x,\tau))\psi_{p,R}^{\prime}(t_x,\tau) d\tau, \label{eq_betaR_COV}
\end{align}
where 
\begin{equation} 
\psi_{p,L}(t_{x},\tau) = t_{x} - (t_{x}+1)\psi_{p}\left(-\frac{1+\tau}{2}\right),\ \psi_{p,R}(t_{x},\tau) = t_{x} + (1-t_{x})\psi_{p}\left(-\frac{1-\tau}{2}\right), 
\end{equation}
\begin{equation}
\psi_{p,L}'(t_{x},\tau) =\frac{t_{x}+1}{2}\psi_{p}'\left(-\frac{1+\tau}{2}\right),\ \psi_{p,R}'(t_{x},\tau) =\frac{1-t_{x}}{2}\psi_{p}'\left(-\frac{1-\tau}{2}\right),
\end{equation}
for $\tau \in [-1,1]$, $p \ge 2$. Note that the derivative $\psi_{p,L}',\psi_{p,R}'$ are with respect to variable $\tau$. If $x=a$, observe that only integral $\beta_{k,L}$ remains and for $x=b$, only $\beta_{k,R}$ remains. The calculations can be divided in cases $\alpha=0$ and $0<\alpha<1$, which we write in the following two lemmas. 

\begin{lemma}\label{lemma_beta_approx_for_log}
	Let $x\in \Omega$ and $g_{\alpha}(|x-y|) = \log{|x-y|}$, that is $\alpha=0$. Then, for $k\geq 0$ 
	\begin{equation}
	\left|\beta_{k}(x) - \tilde{\beta}_{k}(x)\right| \le  C \frac{h \log{n_{\beta}}}{n_{\beta}^{2p}},
    \end{equation}
    where $p\ge 2$ is an integer, $h$ is length of patch $\Omega$, $n_{\beta}$ denotes the number of quadrature nodes in the FF-Rule approximation both $\tilde{\beta}_{k,L}(x)$ and $\tilde{\beta}_{k,R}(x)$, and $C$ is some constant independent of $h$ and $n_{\beta}$.
\end{lemma}
\begin{proof}
We first estimate the error $|\beta_{k,L}(x)-\tilde{\beta}_{k,L}(x)|$ for $x\in\Omega \setminus\{a\}$. Utilizing (2) of \Cref{lemma_wp_properties}, decompose the integrand in \eqref{eq_betaL_COV} as a sum of singular and regular parts, and denote them $F^{s}(t_{x},\tau)$ and $F^{r}(t_{x},\tau)$ respectively, that is
\begin{equation}
    g_{\alpha}\left(\frac{h}{2}(t_x+1)\psi_{p}\left(-\frac{1+\tau}{2}\right)\right)T_{k}(\psi_{p,L}(t_x,\tau))\psi^{\prime}_{p,L}(t_x,\tau) = F^{r}(t_{x},\tau)+F^{s}(t_{x},\tau),
\end{equation}
where
\begin{equation}
    F^{r}(t_{x},\tau)=  \frac{h}{2}\left[\log\left(\frac{h}{2}\right) + \log{\left|\frac{t_x+1}{2^{p}}Q_{p}\left(-\frac{1+\tau}{2}\right)\right|}\right]~T_{k}(\psi_{p,L}(t_x,\tau)) \left(\frac{1+t_{x}}{2}\right) \psi^{\prime}_{p}(\tau),
    \end{equation}
and
\begin{equation}	
F^{s}(t_x,\tau) = \frac{h}{2}p\log{(1-\tau)}~T_{k}(\psi_{p,L}(t_x,\tau))\left(\frac{1+t_{x}}{2}\right)\psi^{\prime}_{p}(\tau). \\
\end{equation}
Using (3) of Lemma(\ref{lemma_wp_properties}), $F^{s}$ can be factorized as $ F^{s}(t_{x},\tau) = f^s(\tau)\eta(t_{x},\tau)$, where
\begin{equation*}
	\eta(t_{x},\tau) = \frac{h}{2}\frac{p(t_x+1)}{2^{p}}~T_{k}(\psi_{p,L}(t_x,\tau))~R_{p}\left(-\frac{1+\tau}{2}\right),\quad f^s(\tau) = (1-\tau)^{p-1}~\log{(1-\tau)}. 
\end{equation*} 
Since $Q_{p}$, $T_{k}$ and $\psi_{p}$ are smooth, $F^{r}$ is a smooth function and  $\I_{ff}^{n}[F^{r}]$ converges super-algebraically to the integral of $F^{r}$ using \Cref{theorem_quad_convergence}. Employing (5) of Lemma(\ref{lemma_wp_properties}), $\eta$ is a smooth function and doesn't contribute asymptotically to the singularity at $\tau=1$ in $f^s$. Therefore, an application of \cite[Theorem 6]{criscuolo1990convergence} for $f^s(\tau)$ we obtain the following estimate: 
\begin{equation}
    |\beta_{k,L}(x) - \tilde{\beta}_{k,L}(x)| \le  C \frac{h \log{n_{\beta}}}{n_{\beta}^{2p}}.
\end{equation}
Similarly, to estimate the error in $|\beta_{k,R}(x)-\tilde{\beta}_{k,R}(x)|$ for $x \in \Omega \setminus \{b\}$, utilizing (2) of \Cref{lemma_wp_properties}, one can decompose the integrand \eqref{eq_betaR_COV} as a sum of singular $F^s$ and regular $F^r$ part, from which we get $f^s(\tau) = (1+\tau)^{p-1}\log{(1+\tau)}$, whose integral is same as the integral of $(1-\tau)^{p-1}\log{(1-\tau)}$ on $[-1, 1]$. Therefore, one can derive exactly the same bound as $\beta_{k,R}$ in this case as well.
\end{proof}
\begin{remark}
 In the implementation of the method for the kernels with $0<\alpha <1$, the authors in \cite{bruno_sachan2024numerical} observed a interesting picture related to the order of convergence of the method. In particular, in \cite{bruno_sachan2024numerical}, the authors numerically noticed that for kernels of the form $g_{\alpha}(x,y)=|x-y|^{\alpha}, $ where $0<\alpha <1$, the method achieves high-order convergence (in terms of the regularity of $u$) if $p(1-\alpha) \in \mathbb{N}$; otherwise it yields low order convergence.
\end{remark}
\begin{lemma}\label{lemma_beta_approx_for_alpha_kernel}
	Let $x\in \Omega$ and $g_{\alpha}(|x-y|) = |x-y|^{-\alpha}$, that is $0<\alpha<1$. Then, for $k\geq 0$ 
	\begin{equation}
	\left|\beta_{k}(x) - \tilde{\beta}_{k}(x)\right|  \le C \begin{dcases}
		\frac{h^{1-\alpha}}{n_{\beta}^{\lambda}}, & \text{if } p(1-\alpha)\in \N \\
		\frac{h^{1-\alpha}}{n_{\beta}^{2p(1-\alpha)}}, & \text{if } p(1-\alpha)\notin \N,
	\end{dcases}
    \end{equation}
    where $p\ge 2$ is an integer, $\lambda$ is any arbitrary positive integer, $h$ is the  length of patch the $\Omega$, $n_{\beta}$ denotes the number of quadrature nodes in the FF-Rule approximation both $\tilde{\beta}_{k,L}(x)$ and $\tilde{\beta}_{k,R}(x)$, and $C$ is some constant independent of $h$ and $n_{\beta}$.
\end{lemma}
\begin{proof}
We first estimate the error $|\beta_{k,L}(x)-\tilde{\beta}_{k,L}(x)|$ for $x \in \Omega \setminus \{a\}$. Utilizing (2),(3) of \Cref{lemma_wp_properties}, decompose the integrand in \eqref{eq_betaL_COV} as a product of singular $f^{s}$ and regular $\eta$ part, that is
\begin{equation}
    g_{\alpha}\left(\frac{h}{2}(t_x+1)\psi_{p}\left(-\frac{1+\tau}{2}\right)\right)T_{k}(\psi_{p,L}(t_x,\tau))\psi^{\prime}_{p,L}(t_x,\tau) = f^{s}(\tau)\eta(t_x,\tau),
\end{equation}
where
\begin{equation*}
	f^s(\tau) =(1-\tau)^{p(1-\alpha)-1}, \quad \eta(t_x,\tau) = \left(\frac{h}{2}\frac{t_x+1}{2^{p}}\right)^{1-\alpha} \left|Q_{p}\left(-\frac{1+\tau}{2}\right)\right|^{-\alpha} T_{k}(\psi_{p,L}(t_x,\tau)) R_{p}\left(-\frac{1+\tau}{2}\right). 
\end{equation*}  
The analysis further can be divided into two cases: one where $p(1-\alpha)$ is a natural number, and the other where it is not.
In the first case $f^s(\tau)$ becomes a polynomial, which is smooth in $[-1,1]$. From \Cref{theorem_quad_convergence} it follows that FF-Rule converges super-algebraically for such functions. If $p(1-\alpha)$ is not an integer, 
employing (5) of \Cref{lemma_wp_properties}, $\eta$ is a smooth function and doesn't contribute asymptotically to the singularity at $\tau=1$ in $f^s$. Therefore, using \cite[Theorem 1]{kutz1984asymptotic} for $f^s(\tau) =  (1-\tau)^{\sigma}$, with $\sigma = p(1-\alpha)-1$, we have the following estimate for $x \in \Omega \setminus \{a\}$
\begin{equation}
\left|\beta_{k,L}(x) - \tilde{\beta}_{k,L}(x)\right| \le C  \frac{h^{1-\alpha}}{n^{2p(1-\alpha)}}.
\end{equation}
Similarly, we can show that $\left|\beta_{k,R}(x) - \tilde{\beta}_{k,R}(x)\right|$ satisfy the same estimate. 
\end{proof}
\noindent
We now give an estimate for $E_{3}^{n}$ \eqref{eq_error3} in the following theorem.

\begin{theorem}\label{theorem_singular_1d_error_third_term}
	If $u \in X^{m}(\Omega)$ then
	\begin{equation} E_{3}^{n}[u](x) \le C \begin{dcases}
		 \frac{h\log{n_{\beta}}}{n_{\beta}^{2p}}, & \text{if } \alpha=0 \\
		\frac{h^{1-\alpha}}{n_{\beta}^{\lambda}}, & \text{if } 0<\alpha<1 ,\,  p(1-\alpha) \in \mathbb{N}\\
		\frac{h^{1-\alpha}}{n_{\beta}^{2p(1-\alpha)}}, & \text{if } 0<\alpha<1,\,  p(1-\alpha) \notin \mathbb{N},
	\end{dcases}
    \end{equation}
    where $p\ge 2$ is an integer, $\lambda$ is any arbitrary positive integer, $h$ is length of patch $\Omega$, $n_{\beta}$ denotes the number of quadrature nodes in the FF-Rule approximation of $\beta_{k,L}(x)$ and $\beta_{k,R}(x)$ \eqref{eq_weights_betaLR}, and $C$ is some constant independent of $h$ and $n_{\beta}$.
\end{theorem}

\begin{proof}
	Using the definition of $E_{3}^{n}[u]$ in \eqref{eq_error3} and employing  \Cref{lemma_beta_approx_for_log}, and \Cref{lemma_beta_approx_for_alpha_kernel}, the result follows.
\end{proof}
\subsection{Total error}\label{subsection_total_error}

In this section, we present the total error in the methodology discussed in \ref{method}. First we derive the total error for a single patch $\Omega = [a, b]$ where $h = b-a$ is a constant. In the later case, we derive the error in $\Omega$ with varying patches.  
\begin{theorem}\label{Thm_quad_error_1d}(\textit{Single Patch}) If $u \in X^{m}[a,b]$ then the total error $E_{\mathrm{tot}}^{n}$ in the quadrature for a single patch is given by the following:
	\begin{enumerate}
		\item For $\alpha=0$, 
		\begin{equation}
		    E_{\mathrm{tot}}^{n}[u](x) \le C \left( \frac{(\log{n})^{2}}{n^{m+2}} +  \frac{\log{n_{\beta}}}{n_{\beta}^{2p}}\right).
		\end{equation}
		\item For $0<\alpha<1$, 
		\begin{equation}
		    E_{\mathrm{tot}}^{n}[u](x) \le C \left( \frac{1}{n^{m+2-\alpha}}+ \frac{1}{n_{\beta}^{\lambda}}\right), \text{ where } \lambda = \begin{cases}
			k, & \text{ if } p(1-\alpha)\in \N, k\in \N \\
			2p(1-\alpha), & \text{ if }	p(1-\alpha)\notin \N,
		\end{cases}
		\end{equation}
	\end{enumerate}
	$p$ is the order of the change of variable used in the computation of $\tilde{\beta}_{k}(x)$, $n$ is the number of terms in discrete Chebyshev expansion, $n_{\beta}$ is the number of FF quadrature points used in the approximation of $\beta_{k}(x)$ \eqref{eq_weights_betaLR}, and $C$ is some constant independent of $n$ and $n_{\beta}$. 
\end{theorem}
\begin{proof}	
	Proof follows from the error inequality (\ref{Eq_total_err_sing_int}), and ignoring $h$ in   \Cref{theorem_singular_1d_error_first_term}, \Cref{theorem_singular_1d_error_second_term} and \Cref{theorem_singular_1d_error_third_term}.
\end{proof}
In the above theorem we have discussed error estimates for the singular quadrature on a single patch, now we derive the bounds in the case of varying patches.
\begin{theorem}\label{Thm_Patches_Varying_gen}(\textit{Varying Patches})
Let $P$ denote the number of patches of $\Omega$ with length of each patch bounded by $O(1/P)$. Let $n$ and $n_{\beta}$	denote the highest degree of the Chebyshev polynomial which approximates $u$ (see \eqref{trunc_cheby_expan}) and the number of quadrature points used in the approximation of $\beta_{k}^{\ell}(x)$ (see \eqref{sing_quad}) respectively.  If $u \in X^{m}(\Omega)$, $g_{\alpha}(|x-y|)$ is as defined in equation \eqref{eq_kernel}, and $n>m+1$, then the total error in the proposed quadrature satisfies
\begin{equation}
    E_{\mathrm{tot}}^{n}[u] = \max_{1\le \ell \leq P} E_{\ell,\mathrm{sing}}^{n}[u]
\end{equation}
where
\begin{enumerate}
\item For $\alpha=0$, 
\begin{equation}
E_{\ell,\mathrm{sing}}^{n}[u] \le C \left(\frac{h_{\ell}^{m+2}|\log{h_{\ell}}|}{n^{m+3}}+ \frac{h_{\ell}^{m+2} (\log{n})^2}{n^{m+2}}+ \frac{h_{\ell}\log{n_{\beta}}}{n_{\beta}^{2p}}\right).
\end{equation}
\item For $0<\alpha<1$, 
\begin{equation}
E_{\ell,\mathrm{sing}}^{n}[u] \le C\left(  \frac{h_{\ell}^{m+2-\alpha}}{n^{m+2-\alpha}} +  \frac{h_{\ell}^{1-\alpha}}{n_{\beta}^{\lambda}}\right), \text{ where } \lambda = \begin{cases}
k, & \text{ if } p(1-\alpha)\in \N, k\in \N \\
2p(1-\alpha), & \text{ if }	p(1-\alpha)\notin \N,
\end{cases} 
\end{equation}
\end{enumerate}
$p\geq 2$ is the order of the change of variable used in the computation of $\tilde{\beta}_{k}^{\ell}(x)$, and $C$ is some constant independent of $h_{\ell},n$, and $n_{\beta}$. 
\end{theorem}
\begin{proof}
Without loss of generality, assume that $x\in\Omega_{\ell}$ for some $\ell = 1,\dots,P$, that is, there is only one singular patch $\Omega_{\ell}$. The total error $E_{\mathrm{tot}}^{n}[u](x)$ is the sum of the error in the singular integral $E_{\ell,\mathrm{sing}}^{n}[u](x)$ (which is on $\Omega_{\ell}$) and errors in the regular integral $E_{\mathrm{reg}}^{n}[u](x)$ (on other patches for $\Omega_{j}$, $j \ne \ell$). That is,  \begin{equation}\label{tot_error_unsimplified}
    E_{\mathrm{tot}}^n[u](x) \le \left| \int_{\Omega_{\ell}} g_{\alpha}(|x-y|)u(y)dy - \I_{\ell, \mathrm{sing}}^n[u](x)\right| + \sum_{\substack{j=1\\ j \ne \ell}}^{P}\left| \int_{\Omega_{j}} g_{\alpha}(|x-y|)u(y)dy - \I_{j,\mathrm{reg}}^n[u](x)\right|,
\end{equation}
where $\I_{\ell, \mathrm{sing}}^n[u](x)$ and $\I_{j,\mathrm{reg}}^n[u](x)$ are the singular and regular integral approximations as defined in \eqref{sing_quad} and \eqref{eq_regular_int} respectively. Denoting the first  and the second terms of \eqref{tot_error_unsimplified} by the singular error $E_{\ell,\mathrm{sing}}^{n}[u](x)$ and the regular error $E_{\mathrm{reg}}^{n}[u](x)$ respectively, and employing \Cref{theorem_singular_1d_error_first_term}, \Cref{theorem_singular_1d_error_second_term} and \Cref{theorem_singular_1d_error_third_term} with $h=h_{\ell}$, the error term $E_{\ell,\mathrm{sing}}^{n}[u](x)$ for $0<\alpha<1$ becomes
\begin{equation}
E_{\ell,\mathrm{sing}}^{n}[u](x) \le C\left( \frac{h_{\ell}^{m+2-\alpha}}{n^{m+2-\alpha}}+\frac{h_{\ell}^{1-\alpha}}{n_{\beta}^{\lambda}}\right), \text{ where } \lambda = \begin{cases}
k, & \text{ if } p(1-\alpha)\in \N, k\in \N \\
2p(1-\alpha), & \text{ if }	p(1-\alpha)\notin \N,
\end{cases}
\end{equation}
and for $\alpha = 0$ it becomes
\begin{equation}
E_{\ell,\mathrm{sing}}^{n}[u](x) \le C \left(\frac{h_{\ell}^{m+2}|\log{h_{\ell}}|}{n^{m+3}}+ \frac{h_{\ell}^{m+2} (\log{n})^2}{n^{m+2}}+ \frac{h_{\ell}\log{n_{\beta}}}{n_{\beta}^{2p}}\right).
\end{equation}
We now show that the error in regular integration $E_{\mathrm{reg}}^{n}[u](x)$ will not contribute to the total error $E_{\mathrm{tot}}^{n}[u]$. For the estimation of total regular error $E_{\mathrm{reg}}^{n}[u](x)$, observe that $u \in X^m(\Omega)$, which implies there exists finitely many discontinuities of $u^{(m+2)}$ contained in finitely many patches, say $n_d$ number of patches. Additionally, notice that $u \in C^{m+2}$ on $(P - n_d)$ patches, and using \Cref{theorem_quad_convergence} the regular error $E^n_{\mathrm{reg}}[u](x)$ can be written as 
\begin{equation}
E_{\mathrm{reg}}^{n}[u](x) = (P-n_d) O\left(\frac{h_{j_1}^{m+3}}{n^{m+2}}\right) + n_d O\left(\frac{h_{j_2}^{m+2}}{n^{m+2}}\right),
\end{equation}
where $j_1,j_2 \in \{1,\cdots,P\}$ such that $j_{1}$ and $j_{2}$ are the maximum length of the $(P-n_{d})$ patches (where $u\in C^{m+2}$) and $n_{d}$ patches (where $u^{(m)}$ is piecewise $C^{2}$). Since $h_{j} = O(1/P)$ for each $j$, and letting $h_{j_0}=\max\{h_{j_{1}},h_{j_{2}}\}$, we obtain $E_{\mathrm{reg}}^{n}[u](x) = \bigO\left(h_{j_0}^{m+2}/n^{m+2}\right).$ Thus, the total error $E_{\mathrm{tot}}^{n}[u](x)$ for $x\in \Omega_{\ell}$ becomes
\begin{equation}\label{eq_tot_err_x}
E_{\mathrm{tot}}^{n}[u] = E_{\ell,\mathrm{sing}}^{n}[u] +  \bigO\left(\frac{h_{j_{0}}^{m+2}}{n^{m+2}}\right).
\end{equation}
Since the total error in the proposed quadrature $E_{\mathrm{tot}}^{n}[u]$ is the supremum of the total errors corresponding to each point $x$ which is derived in \eqref{eq_tot_err_x} and observing the fact the patch length $h_{j_{0}}$ corresponds to some $x^{\prime}\in \Omega_{j_0}$, the total error is given by $ E_{\mathrm{tot}}^{n}[u] = \sup_{x\in \Omega} E_{\mathrm{tot}}^{n}[u](x) = \max_{1\leq \ell \leq P} E_{\ell,\mathrm{sing}}^{n}[u]$ as desired.
\end{proof}

\begin{theorem}\label{Thm_Patches_Varying}(\textit{Varying Patches with uniform patch-length})  If $u \in X^{m}(\Omega)$, $\Omega = [a,b]$ and $g_{\alpha}(|x-y|)$ is as defined in equation (\ref{eq_kernel}), then for any number of patches, total error in the proposed quadrature for $n>m+1$ is given by the following.
	\begin{enumerate}[leftmargin=1em]
		\item For $\alpha=0$, 
		\begin{equation}
			E_{\textrm{tot}}^{n}[u] \le C \left(\frac{h^{m+2}|\log{h}|}{n^{m+3}}+ \frac{h^{m+2} (\log{n})^2}{n^{m+2}}+ \frac{h\log{n_{\beta}}}{n_{\beta}^{2p}}\right)
		\end{equation}
		\item For $0<\alpha<1$, 
		\begin{equation}
			E_{\textrm{tot}}^{n}[u] \le C\left(  \frac{h^{m+2-\alpha}}{n^{m+2-\alpha}} +  \frac{h^{1-\alpha}}{n_{\beta}^{\lambda}}\right), \text{ where } \lambda = \begin{cases}
			k, & \text{ if } p(1-\alpha)\in \N, k\in \N \\
			2p(1-\alpha), & \text{ if }	p(1-\alpha)\notin \N,
		\end{cases}
		\end{equation}
	\end{enumerate}
	where $h$ is the uniform length of the patches given by $h=\frac{b-a}{P}$, $p\geq 2$ is the order of the change of variable used in the computation of $\tilde{\beta}_{k}^{l}(x)$, $n_{\beta}$ is the number of quadrature points used in the approximation of $\beta_{k}^{l}(x)$  (see \eqref{sing_quad}), $n$ denotes the highest degree of the Chebyshev polynomial which approximates $u$ (see \eqref{trunc_cheby_expan}), and $C$ is some constant independent of $h_{\ell},n$, and $n_{\beta}$.
\end{theorem}
\begin{proof}
Proof follows from \Cref{Thm_Patches_Varying_gen}. 
\end{proof}

\begin{remark}
The error estimate for fixed number of patches (possibly more than one) is the same as error presented in \Cref{Thm_quad_error_1d}, which is an immediate consequence of \Cref{Thm_Patches_Varying_gen}
\end{remark}
\section{Numerical results}\label{numerics}
This section presents the computational performance of the one-dimensional high-order integration scheme described in \Cref{method}. The theoretical convergence rate presented in \Cref{Thm_quad_error_1d} and \Cref{Thm_Patches_Varying} is numerically demonstrated through a variety of numerical results, which are consistently in agreement with the theoretical order of convergence (toc).
Throughout this section, we denote the number of nodes used in Chebyshev expansion of $u$ in each patch by $n$, quadrature nodes in approximating singular integrals by $n_{\beta}$ given by $n_{\beta} = 4n$, the total number of patches by $P$, the 
total number of grid points in the integration domain by $N$ given by $N=n\times P$, $p$ denotes the order of the change of variable $\psi_{p}$ defined in equation \eqref{w_p(t)}, and the relative error denoted by $\epsilon_{N,\infty}$. The numerical order of convergence (noc) is computed using the formula 
\begin{equation*}
	\text{noc} = \log_{b}\left(\frac{\epsilon_{N,\infty}}{\epsilon_{bN,\infty}}\right) , \text{ where } \epsilon_{N,\infty} = \frac{\max_{1\leq j\leq N} |u_{j}-u_{j}^{ex}|}{\max_{1\leq j\leq N}|u_{j}^{ex}|},
\end{equation*}
and $b = 2,3$. Here, $u_{j}^{ex}$ denotes the exact solution corresponding to the approximated solution $u_{j}$ at the quadrature point $x_j$. As elucidated in \Cref{method}, there are two approaches to achieve the high-order convergence of the method: either by increasing the number of quadrature points in each patch while keeping the number of patches fixed or by increasing the number of patches (hence reducing the size of patches) but keeping the number of integration nodes in each patch fixed. Thus, in what follows, we present the convergence behavior for both approaches as mentioned above. 
\begin{example}\textit{\bf{Convergence illustration for the fixed patch case when $\alpha=0$}}\\
Consider $g_{\alpha}(|x-y|) = \log|x-y|$, $u(y) = y^{m}|y|$ over the domain $\Omega=[-1,1]$.  As per \Cref{Thm_quad_error_1d}, the rate of convergence of our method will be $\text{toc} = \min\{2p,m+2\}$. For $m=3,4$ and $p=2,3,4$ relative errors at various levels of discretization are presented in \Cref{Table_Log_m3andm4}. In \Cref{Table_Log_m3andm4}, one can see the effect of $p$ on the noc for fixed regularity of the integral density $u$. To see the effect of the regularity parameter $m$ of density $u$, we take $p=5$ and approximate the integral  (\ref{conv_eqn}) for various values of $m$. In \Cref{Table_Log_p5_fixed}, errors at various discretization levels are reported for the numerical approximation of (\ref{conv_eqn}) with $m=0,1,\dots, 6$, while keeping $p=5$ in all the case. The Numerical order of convergence reported in these tables agrees with the theoretical order of convergence.  
\begin{table}[!hbt]
	\centering
	\scalebox{0.8}{
		\begin{tabular}{c| c| c| c| c| c| c|c|c|c|c|c|c }
			\hline\hline
			\multirow{3}{*}{n} & 
			\multicolumn{6}{c|}{m=3} & 
			\multicolumn{6}{c}{m=4} \\
			\cline{2-13} &
			\multicolumn{2}{c|}{p=2, toc=4} & 
			\multicolumn{2}{c|}{p=3, toc=5} & 
			\multicolumn{2}{c|}{p=4, toc=5} & 
			\multicolumn{2}{c|}{p=2, toc=4} & 
			\multicolumn{2}{c|}{p=3, toc=6} & 
			\multicolumn{2}{c}{p=4, toc=6}\\ 
			\cline{2-13}
			& $\epsilon_{n,\infty}$ & noc  &  $\epsilon_{n,\infty}$ & noc &  $\epsilon_{n,\infty}$ & noc &  $\epsilon_{n,\infty}$ & noc  &  $\epsilon_{n,\infty}$ & noc &  $\epsilon_{n,\infty}$ & noc \\
			\cline{2-3}
			\hline
			4 & 4.00e-02 & -& 4.44e-02 & -& 4.46e-02 & -& 1.06e-01 & -& 1.08e-01 & -& 9.88e-02 &\\
			8 & 2.74e-04 & 7.19& 2.75e-04 & 7.33& 2.75e-04 & 7.34 & 3.41e-04 & 8.28& 3.40e-04 & 8.31& 3.40e-04 & 8.19\\
			16 & 6.73e-06 & 5.35& 6.73e-06 & 5.35 & 6.73e-06 & 5.35& 8.01e-06 & 5.41& 4.52e-06 & 6.23 & 4.52e-06 & 6.23\\
			32 & 2.60e-07 & 4.70& 1.99e-07 & 5.08& 1.99e-07 & 5.08& 4.85e-07 & 4.05& 7.81e-08 & 5.85& 7.81e-08 & 5.85\\
			64 & 1.75e-08 & 3.89 & 6.14e-09 & 5.02& 6.14e-09 & 5.02 & 3.03e-08 & 4.00& 1.39e-09 & 5.81& 1.39e-09 & 5.81\\
			128 & 1.11e-09 & 3.98& 1.91e-10 & 5.00& 1.91e-10 & 5.00 & 1.90e-09 & 4.00& 2.47e-11 & 5.82& 2.47e-11 & 5.82\\
			256 & 6.94e-11 & 4.00& 5.98e-12 & 5.00& 5.98e-12 & 5.00& 1.19e-10 & 4.00& 4.35e-13 & 5.83 & 4.33e-13 & 5.84\\
			512 & 4.38e-12 & 3.98& 1.87e-13 & 5.00& 1.87e-13 & 5.00 & 7.49e-12 & 3.98& 4.07e-14 & 3.42& 5.50e-14 & 2.98\\
			1024 & 3.50e-13 & 3.65 & 2.97e-14& 2.66 & 2.39e-14 & 2.97 & 5.98e-13 & 3.65&-&-&-&-\\
			\hline\hline
	\end{tabular} }
    \vspace{-2mm}
    \caption{Convergence of the proposed integration scheme (fixed patch approach) to the approximation of integral operator (\ref{conv_eqn}) for the kernel $g_{\alpha}(x-y) = \log|x-y|$ with $u(y) = y^{m}|y|$.}
	\label{Table_Log_m3andm4}
\end{table} 
\begin{table}[!hbt]
	\centering
	\scalebox{0.8}{
		\begin{tabular}{c| c| c| c| c| c| c| c| c| c| c| c| c| c|c}
			\hline\hline
			\multirow{2}{*}{n} & 
			\multicolumn{2}{c|}{m=0,toc=2} & 
			\multicolumn{2}{c|}{m=1,toc=3} & 
			\multicolumn{2}{c|}{m=2,toc=4} & 
			\multicolumn{2}{c|}{m=3,toc=5} & 
			\multicolumn{2}{c|}{m=4,toc=6} & 
			\multicolumn{2}{c|}{m=5,toc=7} & 
			\multicolumn{2}{c}{m=6,toc=8} \\
			\cline{2-15}
			& $\epsilon_{n,\infty}$ & noc  &  $\epsilon_{n,\infty}$ & noc &  $\epsilon_{n,\infty}$ & noc &  $\epsilon_{n,\infty}$ & noc & $\epsilon_{n,\infty}$ & noc  &  $\epsilon_{n,\infty}$ & noc &  $\epsilon_{n,\infty}$ & noc \\
			\hline
4 & 8.74e-02 & -& 3.15e-02 & - & 3.93e-02 & -& 2.88e-02 & -& 6.98e-02 & -& 1.03e-01 & -& 1.18e-01 & -\\
8 & 2.99e-02 & 1.55& 1.23e-03 & 4.68& 1.71e-03 & 4.52& 2.76e-04 & 6.70& 3.39e-04 & 7.69& 1.59e-04 & 9.34& 9.21e-04 & 7.00\\
16 & 9.77e-03 & 1.61& 1.43e-04 & 3.10 & 1.27e-04 & 3.76& 6.73e-06 & 5.36 & 4.52e-06 & 6.23 & 6.55e-07 & 7.92& 3.57e-07 & 11.3\\
32 & 3.07e-03 & 1.67 & 1.77e-05 & 3.01 & 9.67e-06 & 3.71& 1.99e-07 & 5.08& 7.81e-08 & 5.85& 4.51e-09 & 7.18& 1.35e-09 & 8.05\\
64 & 9.30e-04 & 1.72& 2.21e-06 & 3.01 & 7.16e-07 & 3.76& 6.14e-09 & 5.02& 1.39e-09 & 5.81& 3.41e-11 & 7.04 & 5.71e-12 & 7.89\\
128 & 2.74e-04 & 1.77& 2.76e-07 & 3.00 & 5.19e-08 & 3.79& 1.91e-10 & 5.00 & 2.47e-11 & 5.82 & 2.65e-13 & 7.01& 1.80e-13 & 4.99\\
256 & 7.87e-05 & 1.80 & 3.45e-08 & 3.00& 3.70e-09 & 3.81 & 5.98e-12 & 5.00& 4.34e-13 & 5.83& 4.37e-14 & 2.60& 1.38e-13 & 0.38\\
512 & 2.22e-05 & 1.82& 4.31e-09 & 3.00 & 2.59e-10 & 3.83 & 1.87e-13 & 5.00& 4.07e-14 & 3.41&-&-&-&-\\
1024 & 6.20e-06 &1.84 & 5.38e-10& 3.00 & 1.80e-11 & 3.85 & 2.27e-14 & 3.04& -&- &-&-&-&-\\ 
	\hline\hline
	\end{tabular} }
    \vspace{-2mm}
    \caption{Convergence of the proposed integration scheme (fixed patch approach) to the approximation of integral operator (\ref{conv_eqn}) for the kernel $g_{\alpha}(x-y) = \log|x-y|$ with $u(y) = y^{m}|y|, m=0,1,\dots,6$, and $p=5$.}
	\label{Table_Log_p5_fixed}
\end{table} 
\end{example}
\begin{example}{\textbf{Convergence illustration for the patches varying case when $\alpha=0$}}\\
Consider $g_{\alpha}(x-y)=\ln|x-y|$ with two integral densities $u(y) = y^{2}|y|+1 $ and $u(y) = y^{3}|y|$. \Cref{Table_Log_patches_vary_low_order_m2andm3} presents the numerical results of this experiment for $p=2$ and $p=3$ with $n=16$. One can see that, for the above mentioned $u$, the method yields only first-order convergence which is explained by \Cref{Thm_Patches_Varying}. Specifically, if $u\in X^{m}$ and we choose $p$ such that the term $\log(n)/n^{2p}$ is small in comparison to other error terms, then the method will yield order of convergence $m+2$ demonstrated in \Cref{Table_Log_patches_vary_high_order_m2andm3} for $p=4,5$ and $6$. 
\begin{table}[!hbt]	
    \centering\scalebox{0.8}{\begin{tabular}{c| c|c| c|c| c|c| c|c}
			\hline\hline
			\multirow{3}{*}{$P\times n$} & 
			\multicolumn{4}{c|}{m=2} & 
			\multicolumn{4}{c}{m=3} \\
			 \cline{2-9} &
			\multicolumn{2}{c|}{p=2, toc=1} & 
			\multicolumn{2}{c|}{p=3, toc=1} & 
			\multicolumn{2}{c|}{p=2, toc=1} & 
			\multicolumn{2}{c}{p=3, toc=1} \\ 
			\cline{2-9}						& $\epsilon_{N,\infty}$ & noc  &  $\epsilon_{N,\infty}$ & noc&$\epsilon_{N,\infty}$ & noc  &  $\epsilon_{N,\infty}$ & noc \\ 
						\hline
						$1\times 16$ & 4.03e-04 & - & 4.03e-04& - & 2.74e-04 & -& 2.74e-04 & -\\
						$3\times 16$ & 7.18e-06 & 3.67 & 7.15e-06 & 3.67& 1.08e-06 & 5.04& 1.08e-06 & 5.04\\
						$9\times 16$ & 1.25e-07 & 3.69 & 1.15e-07& 3.76 & 3.14e-08& 3.22 & 4.44e-09 & 5.00\\ 
						$27\times 16$ & 6.53e-09 & 2.69 & 1.74e-09 & 3.82& 1.05e-08& 1.00& 7.16e-11 & 3.76\\
						$81\times 16$ & 2.18e-09 & 1.00 & 1.85e-11 & 4.41&3.51e-09 & 1.00& 2.39e-11 & 1.00\\
						$243\times 16$ & 7.26e-10 & 1.00 & 4.95e-12 &1.20& 1.17e-09 & 1.00& 7.97e-12 & 1.00\\ 
						$729\times 16$ & 2.42e-10& 1.00 & 1.65e-12 & 1.00& 3.90e-10 & 1.00&2.66e-12 & 1.00\\
						$2187\times 16$ & 8.06e-11 & 1.00 & 5.53e-13 & 0.99& 1.30e-10 & 1.00& 8.90e-13 & 1.00\\
						\hline\hline
				\end{tabular}}
                \vspace{-2mm}
				\caption{Convergence of the proposed integration scheme (varying patch approach) to the approximation of integral operator (\ref{conv_eqn}) for the kernel $g_{\alpha}(x-y) = \log|x-y|$ with the integral densities $u(y) = y^{m}|y|+1$ ($m=2$) and $u(y)=y^{m}|y|$ ($m=3$).}
				\label{Table_Log_patches_vary_low_order_m2andm3}
			\end{table}	
\begin{table}[!hbt]
	\centering\scalebox{0.8}{\begin{tabular}{c| c|c| c|c| c|c| c|c| c|c|c|c}
		\hline\hline
		\multirow{3}{*}{$P\times n$} & 
		\multicolumn{6}{c|}{m=2} & 
		\multicolumn{6}{c}{m=3} \\
		\cline{2-13}
		&
		\multicolumn{2}{c|}{p=4, toc=4} & 
		\multicolumn{2}{c|}{p=5, toc=4} & 
		\multicolumn{2}{c|}{p=6, toc=4} & 
		\multicolumn{2}{c|}{p=4, toc=5} &
		\multicolumn{2}{c|}{p=5, toc=5} & 
		\multicolumn{2}{c}{p=6, toc=5} \\ 
		\cline{2-13}
		& $\epsilon_{n,\infty}$ & noc  &  $\epsilon_{n,\infty}$ & noc &  $\epsilon_{n,\infty}$ & noc &  $\epsilon_{n,\infty}$ & noc & $\epsilon_{n,\infty}$ & noc  &  $\epsilon_{n,\infty}$ & noc\\ 
		\hline
			$1\times 16$& 4.03e-04 & - & 4.03e-04 & - & 4.03e-04 & -& 2.74e-04 & - & 2.74e-04 & - & 2.74e-04 & -  \\
			$3\times 16$ & 7.15e-06 & 3.67 &7.15e-06 & 3.67 & 7.15e-06 & 3.67 & 1.08e-06 & 5.04 &1.08e-06 & 5.04 &1.08e-06 & 5.04\\
			$9\times 16$ & 1.15e-07 & 3.76 & 1.15e-07 & 3.76 & 1.15e-07 & 3.76& 4.44e-09 & 5.00 & 4.44e-09 & 5.00& 4.44e-09 & 5.00 \\ 
			$27\times 16$ & 1.76e-09 & 3.81 & 1.76e-09 & 3.81 & 1.76e-09 & 3.81 & 1.82e-11 & 5.00& 1.82e-11 & 5.00& 1.82e-11 & 5.00 \\
			$81\times 16$ &2.60e-11 & 3.84 & 2.59e-11& 3.84 & 2.59e-11 & 3.84&1.85e-13 & 4.18 & 7.49e-14& 5.00 & 7.49e-14& 5.00 \\
			$243\times 16$& 3.91e-13 & 3.82 & 3.72e-13 & 3.86 & 3.72e-13& 3.86 & 6.17e-14 & 1.00 & 3.20e-15 & 2.87 & 2.70e-15& 3.03\\ 
		$729\times 16$ & 3.12e-14 & 2.30 & 6.27e-15 & 3.72 & 6.48e-15 & 3.69& 2.24e-14 & 0.92 & - & - & - & - \\
			$2187\times 16$ & 1.65e-14 & 0.58 & - & -& - & -& 1.25e-14 & 0.53 & - & -& - & -\\
			\hline\hline
	\end{tabular}}
    \vspace{-2mm}
	\caption{Convergence of the proposed integration scheme (varying patch approach) to the approximation of integral operator (\ref{conv_eqn}) for the kernel $g_{\alpha}(x-y) = \log|x-y|$ with the integral densities $u(y) = y^{m}|y|+1$ ($m=2$) and $u(y)=y^{m}|y|$ ($m=3$). Increasing the value of $p$ leads to a high-order convergence rate.}
    \label{Table_Log_patches_vary_high_order_m2andm3}
\end{table}
\end{example}	

\begin{example}\textbf {Convergence illustration for fixed patch approach when $0<\alpha<1$}
\label{Example_patches_fixed_alpha0.75and0.9_low_order}\\
Consider the kernel $g_{\alpha}(x-y)=|x-y|^{-\alpha}$, $\alpha=0.75,0.9$ for $u(y) = y^{3}|y|\in X^{3}[-1,1]$ and $u(y) = y^{4}|y|\in X^{4}[-1,1]$ with $p=2,3,6$.
As per \Cref{Thm_quad_error_1d}, the order of convergence of the method will rely upon three parameters $\alpha$, $m$, and $p$. To be precise, if $p(1-\alpha)\in \N$ then $\text{toc} = m+2-\alpha$, otherwise $\text{toc} = \min\{m+2-\alpha, 2p(1-\alpha)\}$.  

For $p=2,3,6$, the term $p(1-\alpha) \notin \mathbb{N}$, therefore in this case $\text{toc} = \min\{m+2-\alpha, 2p(1-\alpha)\}$. Since $2p(1-\alpha)<m+2-\alpha$ for $m=3,4$ and $p=2,3,6$ with $\alpha=0.75,0.9$, the toc is dominated by the quantity $2p(1-\alpha)$, which is exactly demonstrated numerically in \Cref{Table_alpha0.75_and_m3_patches_fixed_low_order} and \Cref{Table_alpha0.75_and_m4_patches_fixed_low_order} for $m=3$ and $m=4$ respectively. 

\begin{table}[hbt!]
	\centering
	\scalebox{0.8}{
	\begin{tabular}{c|c|c|c|c|c|c|c|c|c|c|c|c}
		\hline\hline
		\multirow{3}{*}{n} & 
		\multicolumn{6}{c|}{$\alpha=0.75$, $u(y)=y^3|y|$} & 
		\multicolumn{6}{c}{$\alpha = 0.9$, $u(y)=y^3|y|$} \\
		\cline{2-13} &
		\multicolumn{2}{c|}{p=2, toc=1} & 
		\multicolumn{2}{c|}{p=3, toc=1.5} & 
		\multicolumn{2}{c|}{p=6, toc=3} & 
		\multicolumn{2}{c|}{p=2, toc=0.4} & 
		\multicolumn{2}{c|}{p=3, toc=0.6} & 
		\multicolumn{2}{c}{p=6, toc=1.2}\\ 
		\cline{2-13}
		&$\epsilon_{n,\infty}$ & noc &  $\epsilon_{n,\infty}$ & noc &  $\epsilon_{n,\infty}$ & noc &  $\epsilon_{n,\infty}$ & noc &  $\epsilon_{n,\infty}$ & noc &  $\epsilon_{n,\infty}$ & noc\\
			\cline{2-3}
			\hline
4 & 5.18e-02 &- & 6.12e-03 & -& 1.28e-02 & -& 3.48e-01 & -& 2.01e-01 & -& 1.93e-02 & -\\
8 & 2.92e-02 & 0.83& 1.87e-03 & 1.71& 1.29e-03 & 3.32& 2.71e-01 & 0.36& 1.39e-01 & 0.53& 1.10e-02 & 0.80 \\
16 & 1.47e-02 & 0.99& 7.14e-04 & 1.39& 5.22e-05 & 4.62& 2.08e-01 & 0.38& 9.26e-02 & 0.58& 4.87e-03 & 1.18\\
32 & 7.37e-03 & 0.99& 2.55e-04 & 1.48& 6.15e-06 & 3.08 & 1.57e-01 & 0.41& 6.08e-02 & 0.61 & 2.11e-03 & 1.21 \\
64 & 3.69e-03 & 1.00 & 9.04e-05 & 1.50 & 7.63e-07 & 3.01& 1.19e-01 & 0.40& 4.02e-02 & 0.60& 9.21e-04 & 1.20\\
128 & 1.84e-03 & 1.00 & 3.20e-05 & 1.50& 9.52e-08 & 3.00 & 9.01e-02 & 0.40& 2.65e-02 & 0.60 & 4.01e-04 & 1.20 \\
256 & 9.22e-04 & 1.00& 1.13e-05 & 1.50& 1.19e-08 & 3.00 & 6.83e-02 & 0.40& 1.75e-02 & 0.60& 1.74e-04 & 1.20\\
512 & 4.61e-04 & 1.00& 4.00e-06 & 1.50 & 1.49e-09 & 3.00  & 5.17e-02 & 0.40& 1.15e-02 & 0.60 & 7.59e-05 & 1.20\\
1024 & 2.30e-04 & 1.00& 1.41e-06 & 1.50& 1.86e-10 & 3.00 & 3.92e-02 & 0.40 & 7.60e-03 & 0.60 & 3.30e-05 & 1.20 \\ \hline\hline
	\end{tabular} } 
    \vspace{-2mm}
    \caption{Convergence of the proposed integration scheme (fixed patch approach with $P=1$) to the approximation of integral operator (\ref{conv_eqn}) for the kernel $g_{\alpha}(x-y) = |x-y|^{-\alpha}$, $\alpha=0.75,0.9$ with $u(y) = y^{3}|y|$.}
	\label{Table_alpha0.75_and_m3_patches_fixed_low_order}
\end{table} 	
\begin{table}[hbt!]
	\centering
	\scalebox{0.8}{
		\begin{tabular}{c| c| c| c| c|c|c|c|c|c|c|c|c }
			\hline\hline
			\multirow{3}{*}{n} & 
			\multicolumn{6}{c|}{$\alpha=0.75$,$u(y)=y^4|y|$} & 
			\multicolumn{6}{c}{$\alpha = 0.9$,$u(y)=y^4|y|$} \\
			\cline{2-13} &
			\multicolumn{2}{c|}{p=2, toc=1} & 
			\multicolumn{2}{c|}{p=3, toc=1.5} & 
			\multicolumn{2}{c|}{p=6, toc=3} & 
			\multicolumn{2}{c|}{p=2, toc=0.4} & 
			\multicolumn{2}{c|}{p=3, toc=0.6} & 
			\multicolumn{2}{c}{p=6, toc=1.2} \\ 
			\cline{2-13}
			&  $\epsilon_{n,\infty}$ & noc &  $\epsilon_{n,\infty}$ & noc  &  $\epsilon_{n,\infty}$ & noc &  $\epsilon_{n,\infty}$ & noc &  $\epsilon_{n,\infty}$ & noc &  $\epsilon_{n,\infty}$ & noc\\
			\cline{2-3}
			\hline
		4 & 1.86e-02 &-& 2.91e-02 & -& 7.72e-03 & -& 3.33e-01 & - & 1.89e-01 &- & 1.68e-02 & - \\
		8 & 2.83e-02 & -0.61& 2.00e-03 & 3.86 & 2.46e-04 & 4.97& 2.70e-01 & 0.30 & 1.38e-01 & 0.46 & 1.07e-02 & 0.65\\
		16& 1.46e-02 & 0.95 & 7.13e-04 & 1.49 & 5.05e-05 & 2.28& 2.07e-01 & 0.38& 9.23e-02 & 0.58 & 4.85e-03 & 1.14\\
		32& 7.22e-03 & 1.02 & 2.50e-04 & 1.51 & 5.97e-06 & 3.08 & 1.56e-01 & 0.41 & 6.05e-02 &0.61& 2.10e-03 & 1.21\\
		64 & 3.59e-03 & 1.01& 8.80e-05 & 1.51& 7.41e-07 & 3.01 & 1.19e-01 & 0.39& 4.00e-02 & 0.59& 9.18e-04 & 1.19\\
		128 & 1.79e-03 & 1.00& 3.11e-05 & 1.50 & 9.26e-08 & 3.00& 8.98e-02 & 0.40& 2.64e-02 & 0.60  & 3.99e-04 & 1.20\\
		256 & 8.97e-04 & 1.00 & 1.10e-05 & 1.50 & 1.16e-08 & 3.00& 6.81e-02 & 0.40 & 1.74e-02 & 0.60 & 1.74e-04 & 1.20\\
		512 & 4.49e-04 & 1.00 & 3.90e-06 & 1.50 & 1.45e-09 & 3.00& 5.16e-02 & 0.40 & 1.15e-02 & 0.60& 7.57e-05 & 1.20\\
		1024 & 2.24e-04 & 1.00& 1.38e-06 & 1.50 & 1.81e-10 & 3.00 & 3.91e-02 & 0.40& 7.58e-03 &0.60 & 3.29e-05 & 1.20\\ 	
			\hline\hline
	\end{tabular} } 
    \vspace{-2mm}
    \caption{Convergence of the proposed integration scheme (fixed patch approach) to the approximation of integral operator (\ref{conv_eqn}) for the kernel $g_{\alpha}(x-y) = |x-y|^{-\alpha}$, $\alpha=0.75,0.9$ with $u(y) = y^{4}|y|$.}
	\label{Table_alpha0.75_and_m4_patches_fixed_low_order}
\end{table} 
Therefore, in this case to achieve optimal order of convergence in terms of $m$, the choices of the order of change of variable $p$ for $\alpha=0.75$ are $p\in 4\N$ and for $\alpha=0.9$ are $p\in 10\N$. For instance, if $m=3$ then the toc is $4.25$ and $4.10$ for $\alpha=0.75$ and $\alpha=0.9$ respectively. We numerically describe this behavior of the proposed method in \Cref{Table_alpha0.75and0.9_and_m3_patches_fixed_high_order}. Similar behavior is noticed if $m=4$ for $\alpha=0.75$ and $\alpha=0.9$ in \Cref{Table_alpha0.75and0.9_and_m4_patches_fixed_high_order}.
\begin{table}[hbt!]
	\centering
	\scalebox{0.8}{
		\begin{tabular}{c| c| c| c| c| c| c| c| c|c|c|c|c }
			\hline\hline
			\multirow{3}{*}{n} & 
			\multicolumn{6}{c|}{$\alpha=0.75$, $u(y)=y^3|y|$} & 
			\multicolumn{6}{c}{$\alpha = 0.9$, $u(y)=y^3|y|$} \\
			\cline{2-13} &
			\multicolumn{2}{c|}{p=4, toc=4.25} & 
			\multicolumn{2}{c|}{p=8, toc=4.25} & 
			\multicolumn{2}{c|}{p=12, toc=4.25} & 
			\multicolumn{2}{c|}{p=10, toc=4.10} &
			\multicolumn{2}{c|}{p=20, toc=4.10} & 
			\multicolumn{2}{c}{p=30, toc=4.10} \\ 
			\cline{2-13}
			& $\epsilon_{n,\infty}$ & noc  &  $\epsilon_{n,\infty}$ & noc &  $\epsilon_{n,\infty}$ & noc &  $\epsilon_{n,\infty}$ & noc  &  $\epsilon_{n,\infty}$ & noc &  $\epsilon_{n,\infty}$ & noc\\
			\cline{2-3}
			\hline
		4& 1.54e-02 &-& 6.04e-02 & -& 1.72e-01 & -& 2.82e-02 & -& 8.43e-02 & - & 7.71e-02 & -\\
		 8& 1.54e-04 & 6.64 & 5.02e-03 & 3.59 & 5.08e-03 & 5.08& 2.47e-03 & 3.51 & 9.02e-03 & 3.23& 1.42e-02 & 2.44\\
		 16 & 4.64e-06 & 5.06& 3.32e-05 & 7.24 & 5.35e-05 & 6.57& 1.71e-05 & 7.18& 2.01e-05 & 8.81 & 5.18e-05 & 8.10\\
		 32 & 2.27e-07 & 4.35& 5.07e-07 & 6.03& 1.61e-06 & 5.06 & 3.77e-07 & 5.50 & 6.59e-07 & 4.93 & 1.31e-06 & 5.30\\
		 64 & 1.18e-08 & 4.27 & 2.62e-08 & 4.27 & 4.57e-08 & 5.14 & 9.55e-09 & 5.30 & 2.45e-08 & 4.75 & 3.86e-08 & 5.09\\
		 128 & 6.17e-10 & 4.25 & 6.93e-10 & 5.24& 1.44e-09 & 4.98 & 3.61e-10 & 4.73 & 8.28e-10 & 4.89& 1.25e-09 & 4.94\\
		 256 & 3.24e-11 & 4.25 & 3.24e-11 & 4.42& 6.13e-11 & 4.56 & 2.10e-11 & 4.10 & 2.47e-11 & 5.06& 3.96e-11 & 4.99\\
		 512 & 1.70e-12 & 4.25& 1.70e-12 & 4.25 & 1.95e-12 & 4.98 & 1.23e-12 & 4.10& 1.23e-12 & 4.34 & 1.24e-12 & 4.99\\
		 1024 & 1.68e-13 & 3.34& 1.79e-13 & 3.25& 2.14e-13 & 3.18& 2.45e-13 & 2.32 & 2.20e-13 & 2.48& 2.05e-13 & 2.60\\ 	 
			\hline\hline
	\end{tabular} } 
    \vspace{-2mm}
    \caption{Convergence of the proposed integration scheme (fixed patch approach) to the approximation of integral operator (\ref{conv_eqn}) for the kernel $g_{\alpha}(x-y) = |x-y|^{-\alpha}$, $\alpha=0.75,0.9$ with $u(y) = y^{3}|y|$.}
\label{Table_alpha0.75and0.9_and_m3_patches_fixed_high_order}
\end{table} 
\begin{table}[hbt!]
	\centering
	\scalebox{0.8}{
		\begin{tabular}{c| c| c| c| c| c| c| c| c|c|c|c|c }
			\hline\hline
			\multirow{3}{*}{n} & 
			\multicolumn{6}{c|}{$\alpha=0.75$, $u(y)=y^4|y|$} & 
			\multicolumn{6}{c}{$\alpha = 0.9$, $u(y)=y^4|y|$} \\
			\cline{2-13} &
			\multicolumn{2}{c|}{p=4, toc=5.25} & 
			\multicolumn{2}{c|}{p=8, toc=5.25} & 
			\multicolumn{2}{c|}{p=12, toc=5.25} & 
			\multicolumn{2}{c|}{p=10, toc=5.10} &
			\multicolumn{2}{c|}{p=20, toc=5.10} & 
			\multicolumn{2}{c}{p=30, toc=5.10} \\ 
			\cline{2-13}
			& $\epsilon_{n,\infty}$ & noc  &  $\epsilon_{n,\infty}$ & noc &  $\epsilon_{n,\infty}$ & noc &  $\epsilon_{n,\infty}$ & noc  &  $\epsilon_{n,\infty}$ & noc &  $\epsilon_{n,\infty}$ & noc\\
			\cline{2-3}
			\hline
		4 & 3.17e-02 &-& 4.29e-02 & -& 1.21e-01 &-& 2.21e-02 & -& 5.02e-02 & - & 2.64e-02 & -\\
		8 & 1.21e-04 & 8.04& 4.41e-03 & 3.28& 1.03e-02 & 3.56& 2.96e-03 & 2.90 & 8.07e-03 & 2.64 & 1.92e-02 & 0.46\\
		16 & 1.43e-06 & 6.40& 1.66e-06 & 11.4 & 1.74e-05 & 9.20& 2.75e-06 & 10.1 & 1.39e-05 & 9.18& 2.41e-04 & 6.32\\
		32 & 3.47e-08 & 5.37& 3.47e-08 & 5.58 & 1.36e-07 & 7.01& 2.15e-08 & 7.00& 1.18e-07 & 6.88 & 3.07e-07 & 9.62\\
		64 & 8.95e-10 & 5.28& 8.95e-10 & 5.28& 1.92e-09 & 6.14 & 4.47e-10 & 5.59& 1.88e-09 & 5.97& 3.94e-09 & 6.29\\
		128 & 2.34e-11 & 5.26 & 2.34e-11 & 5.26 & 2.64e-11 & 6.18& 1.30e-11 & 5.11 & 2.93e-11 & 6.00& 6.35e-11 & 5.95\\
		256 & 6.15e-13 & 5.25 & 6.14e-13 & 5.25& 6.15e-13 & 5.43& 3.78e-13 & 5.10& 4.54e-13 & 6.01 & 9.89e-13 & 6.00\\
		512 & 5.83e-14 & 3.40 & 5.79e-14 & 3.41& 5.24e-14 & 3.55& 7.03e-14 & 2.43 & 6.54e-14 & 2.79& 8.17e-14 & 3.60\\
			\hline\hline
	\end{tabular} } 
    \vspace{-2mm}
    	\caption{Convergence of the proposed integration scheme (fixed patch approach) to the approximation of integral operator (\ref{conv_eqn}) for the kernel $g_{\alpha}(x-y) = |x-y|^{-\alpha}$, $\alpha=0.75$, $0.9$ with $u(y) = y^{4}|y|$.}	\label{Table_alpha0.75and0.9_and_m4_patches_fixed_high_order}
\end{table} 
	\end{example}
\begin{example}\textbf {Convergence illustration for varying patch approach when $0<\alpha<1$}\\	
	Consider the kernel $g_{\alpha}(x-y) = |x-y|^{-\alpha}$, $\alpha=0.75$ and $u(y) = y^4|y|+y+1 \in X^4[-1,1]$. According to \Cref{Thm_Patches_Varying}, if $p(1-\alpha)$ is not an integer, then the toc is $1-\alpha = 0.25$, which indicates poor convergence. For instance, if $p = 3,5,6,7$, then the toc becomes $0.25$, which is consistent with the numerical experiment shown in \Cref{Table_alpha0.75_and_m4_patches}. However, if $p(1-\alpha)$ is an integer, then the toc depends on the regularity of $u$ which is $\text{toc}=m+2-\alpha$ as analytically estimated in \Cref{Thm_Patches_Varying}. This phenomenon can be observed in the same table when $p=4,8$, where the noc matches exactly with the toc.     
	\begin{table}[hbt!]
		\centering\scalebox{0.8}{\begin{tabular}{c|c|c|c|c|c|c|c|c|c|c}
				\hline\hline
				\multirow{2}{*}{P} & 
				\multicolumn{2}{c|}{p=4(toc=5.25)} & 
				\multicolumn{2}{c|}{p=5(toc=0.25)} & 
				\multicolumn{2}{c|}{p=6(toc=0.25)} & 
				\multicolumn{2}{c|}{p=7(toc=0.25)} & 
				\multicolumn{2}{c}{p=8(toc=5.25)} \\ 
				\cline{2-11}
				& $\epsilon_{\infty}$ & noc  &  $\epsilon_{\infty}$ & noc &  $\epsilon_{\infty}$ & noc &  $\epsilon_{\infty}$ & noc &  $\epsilon_{\infty}$ & noc \\
				\hline
				$1$ & $5.40e-05$ & $-$ & $5.72e-05$ & $-$ & $5.51e-05$ & $-$ & $5.41e-05$ & $-$ & $5.40e-05$ & $-$\\
				$3$ & $1.47e-07$ & $5.37$ & $3.28e-06$ & $2.60$ & $1.17e-06$ & $3.51$ & $2.36e-07$ & $4.95$ & $1.47e-07$ & $5.37$\\
				$9$ & $3.83e-10$ & $5.42$ & $2.66e-06$ & $0.19$ & $9.50e-07$ & $0.19$ & $1.52e-07$ & $0.40$ & $3.83e-10$ & $5.42$\\
				$27$ & $1.13e-12$ & $5.30$ & $2.19e-06$ & $0.18$ & $7.82e-07$ & $0.18$ & $1.25e-07$ & $0.18$& $1.13e-12$ & $5.30$\\
				$81$ & $6.00e-15$ & $4.77$ & $1.80e-06$ & $0.18$ & $6.42e-07$ & $0.18$ & $1.03e-07$ & $0.18$& $5.69e-15$ & $4.82$\\
				$243$ & $-$ & $-$ & $1.46e-06$ & $0.19$ & $5.19e-07$ & $0.19$ & $8.29e-08$ & $0.19$ & $-$ & $-$\\
				$729$ &  $-$ & $-$  & $1.17e-06$ & $0.20$ & $4.17e-07$ & $0.20$ & $6.65e-08$ & $0.20$ & $-$ & $-$\\
				$2187$ &  $-$ & $-$ & $9.27e-07$ & $0.21$ & $3.31e-07$ & $0.21$ & $5.28e-08$ & $0.21$ & $-$ & $-$\\ 
				\hline\hline
		\end{tabular}}
        \vspace{-2mm}
		\caption{Convergence of the proposed integration scheme (varying patch approach) to the approximation of integral operator (\ref{conv_eqn}) for the kernel $g_{\alpha}(x-y) = |x-y|^{-\alpha}$, $\alpha=0.75$ with  $u(y) = y^{4}|y|+y+1$.}
\label{Table_alpha0.75_and_m4_patches}
	\end{table}
\end{example}
\section{Applications to surface scattering problem}\label{sec:scattering}
The scattered wave $u^{s}$ which arises in the acoustic scattering by sound soft obstacle $\Omega$, is given by the unique solution of the following PDE. 
\begin{equation}\label{Surf_Helmholtz_2d}
	\left\{
	\begin{array}{ll}
		\Delta u^{s}(\x) + \kappa^{2}u^{s}(\x) &= 0,\quad \x \in \R^{2} \setminus \Omega, \\
		u^{s}(\x) &= -u^{i}(\x), \quad \x \in \partial\Omega,\\
		\lim\limits_{r\rightarrow \infty} \sqrt{r}\left( \frac{\partial u^{s}}{\partial r} - i\kappa u^{s}\right) &= 0 ,
	\end{array}
	\right.
\end{equation}
where $u^{i}$ is the incident wave, $\kappa$ is the wave number, $\Omega\subset \R^{2}$ is a bounded domain, $\partial\Omega$ the boundary of $\Omega$, $r=\|\x\|$ and $i=\sqrt{-1}$. The obstacle $\Omega$ is of class $C^{2}$ and the fact, scattered wave satisfies the Sommerfeld radiation condition guarantee the existence of the unique solution to the above problem (\ref{Surf_Helmholtz_2d}) whenever the wavenumber $\kappa$ is positive \cite{colton2013integral}.
An equivalent integral equation formulation of the above exterior Helmholtz problem in terms of the acoustic single and double layer potentials is given by
\begin{equation}\label{Surf_us_exterior}
	u^{s}(\x) = \int\limits_{\partial\Omega} \left[ \frac{\partial g_{\kappa}(\x,\y)}{\partial \nu(\y)} - i\eta g_{\kappa}(\x,\y) \right] \varphi(y), \quad \x \in \R^{2} \setminus \Omega,
\end{equation}
where the density $\varphi$ is the solution of the following integral equation of the second kind
\begin{equation}\label{Surf_density_boundary}
	\frac{\varphi}{2}+(K\varphi)(\x) - i \eta (S\varphi)(\x) = u^{i}(\x),\quad \x \in \partial \Omega,
\end{equation}
and $(S\varphi)(\x)$,$(K\varphi)(\x)$ are the acoustic single and double layer potentials defined as follows.
\begin{equation*}
	(S\varphi)(\x) = \int\limits_{\partial\Omega} g_{\kappa}(\x,\y) \varphi(\y)ds(\y), \text{ and }
	(K\varphi)(\x) = \int\limits_{\partial\Omega} \frac{\partial g_{\kappa}(\x,\y)}{\partial \nu(\y)} \varphi(\y)ds(\y),
\end{equation*}
where $g_{\kappa}(\x,\y) = \frac{i}{4}H_{0}^{1}(\kappa|\x-\y|)$ is the Green function for the free space Helmholtz problem (\ref{Surf_Helmholtz_2d}), $\nu$ is the unit normal vector directed to the exterior of $\Omega$, and $\eta$ is called the coupling constant.
We illustrate the numerical solution of the problem discussed in \eqref{Surf_Helmholtz_2d} through the following experiments. The numerical results presented in this section were produced by means of a C++ implementation of the algorithms described in \Cref{method} on a single core of an Intel i7-11390H processor.
\begin{example}
In this example, we compute the numerical solution of the impenetrable scattering problem \eqref{Surf_us_exterior}-\eqref{Surf_density_boundary} on a circular region by considering a plane wave incidence.
The method proposed in \Cref{method} can effectively solve large-scale problems without suffering from dispersion error. Numerical results for large wavenumbers are computed and presented in \Cref{Table_high_wavenum}. In this table, the acronym ``iter" denotes the number of GMRES iterations required to reach the GMRES tolerance $10^{-10}$, it also illustrates the dispersion-free nature, as the solver's accuracy does not deteriorate with the increase in wavenumber if the number of points per wavelength is maintained. 
\begin{table}[hbt!]
\centering
\scalebox{.8}{
\begin{tabular}{c| c| c| c| c| c| c|c}
\hline \hline
\multirow{2}{*}{$P \times n$} & 
\multirow{2}{*}{$\kappa$} & 
\multirow{2}{*}{Points per $\lambda$} & 
\multirow{2}{*}{$\epsilon_{N,\infty}$} & 
\multirow{2}{*}{iter} & 
\multicolumn{2}{c}{Time (in seconds)}\\
\cline{6-8} 
& & & & &
\multirow{1}{*}{per iter} &
\multirow{1}{*}{pre-comput} & 
\multirow{1}{*}{Total}  \\ 
\cline{1-7}
\hline
$8 \times 15$& 10 & 12 & 8.09e-07 & 16 & 0.0006 & 0.005 & 0.03 \\
$ 16 \times 15$  & 20 &  12 & 2.09e-07 & 21 & 0.001 & 0.02 & 0.08 \\
$ 32 \times 15$ & 40&  12 &  4.66e-08 & 26 & 0.004 & 0.06 & 0.24 \\
$ 64 \times 15 $ & 80 & 12 & 5.74e-08 & 30 & 0.01 & 0.24 & 0.79 \\
$128 \times 15$ & 160&  12& 6.58e-08 & 35 & 0.05 & 0.98 & 3 \\
$256 \times 15$ & 320& 12 & 6.86e-08 & 43 & 0.20 & 3.84 & 13 \\
$ 512 \times 15$ & 640 & 12 & 7.33e-08 & 52 & 0.74 & 15.07 & 57 \\
$ 1024 \times 15$ & 1280 & 12 &4.62e-08 & 64 & 2.90 & 60.08 & 257\\			
$2048 \times 15$ &	2560  & 12 &7.71e-08 & 79 &  - & - & -\\
\hline \hline
\end{tabular} } 
\caption{Performance of the solver for high wave numbers:  Eight-digit accuracy is maintained using 12 points per wavelength. The GMRES tolerance is set to be $10^{-10}$ for all the results.}
\label{Table_high_wavenum}
\end{table} 
\end{example}
\begin{example}
This experiment illustrates the high-order convergence of our numerical integration scheme for star-shaped boundary.
The parametrization of $\partial\Omega$ is 
\begin{equation}
c(t) = r(t)(\cos{t},\sin{t}),\quad r(t) = 1+0.3\cos{(5t)}, ~0\leq t \leq 2\pi.
\end{equation}
For the plane wave incidence with $\kappa = 12$, we compared our numerical solution with the reference solution computed at $21\times 21$ equispaced grids on $[-3,3]^2$. For computing the reference solution, consider $n=20$ on each of the $512$ patches. An iterative solver GMRES took $26$ iterations with $10^{-14}$ tolerance to compute the reference solution at finer grids. 
The high-order convergence in both the approaches, patches fixed and patches varying, is attained and demonstrated in \Cref{Table_Surf_Star1}, respectively.
In \Cref{fig:Star_scat}, we display the surface scattering by star-shaped boundary when the incoming plane wave impinges on the obstacle from the positive x-axis with $\kappa=40$. 
\begin{table}[hbt!]
\begin{minipage}{0.5\linewidth}
\centering
\scalebox{0.8}{
\begin{tabular}{c| c| c| c}
\hline \hline
\multirow{1}{*}{$P\times n$} & 
\multirow{1}{*}{$\epsilon_{N,\infty}$} & 
\multirow{1}{*}{noc} & 
\multirow{1}{*}{time(in seconds)} \\ 
\cline{1-4}
\hline 
$1\times 15$ & 1.77e+00 & - & -\\
$2\times 15$ & 1.00+e00 & 0.82 & 0.01\\
$4\times 15$ & 1.40e-01 & 2.84 & 0.02\\
$8\times 15$ & 1.36e-02 & 3.36 & 0.06\\
$16\times 15$ & 7.21e-04 & 4.24 & 0.14\\
$32\times 15$ & 8.46e-07 & 9.73 & 0.38\\
$64\times 15$ & 4.48e-09 & 7.56 & 1.24\\
$128\times 15$ & 4.27e-13 & 13.36 & 4.39\\
				\hline \hline
		\end{tabular} }
	\end{minipage}
	\begin{minipage}{0.5\linewidth}
\centering
\scalebox{0.8}{
\begin{tabular}{c| c| c| c}
\hline \hline
\multirow{1}{*}{$P\times n$} & 
\multirow{1}{*}{$\epsilon_{N,\infty}$} & 
\multirow{1}{*}{noc} & 
\multirow{1}{*}{time(in seconds)} \\ 
\cline{1-4}
\hline 
$5\times 4$ & 5.29e+00 & - & -\\
$5\times 8$ & 4.89e-01 & 3.44 & 0.01\\
$5\times 16$ & 4.77e-02 & 3.36 & 0.02\\
$5\times 32$ & 1.31e-03 & 5.18 & 0.07\\
$5\times 64$ & 1.49e-05 & 6.45 & 0.26\\
$5\times 128$ & 5.83e-10 & 14.65 & 1.18\\
$5\times 256$ & 1.99e-11 & 4.86 & 6.44\\
$5\times 512$ & 3.51e-12 & 2.51 & 39.92\\
				\hline \hline
		\end{tabular} } 
	\end{minipage}
	 \label{Table_Surf_Star2}
\caption{Illustration of high-order convergence of the proposed algorithm for the surface scattering by star-shaped scatterer with $\kappa = 12$.}
\label{Table_Surf_Star1}
\end{table} 
\begin{figure}[hbt!]
\begin{center}
\begin{subfigure}{0.49\linewidth}
\includegraphics[width=0.85\linewidth]{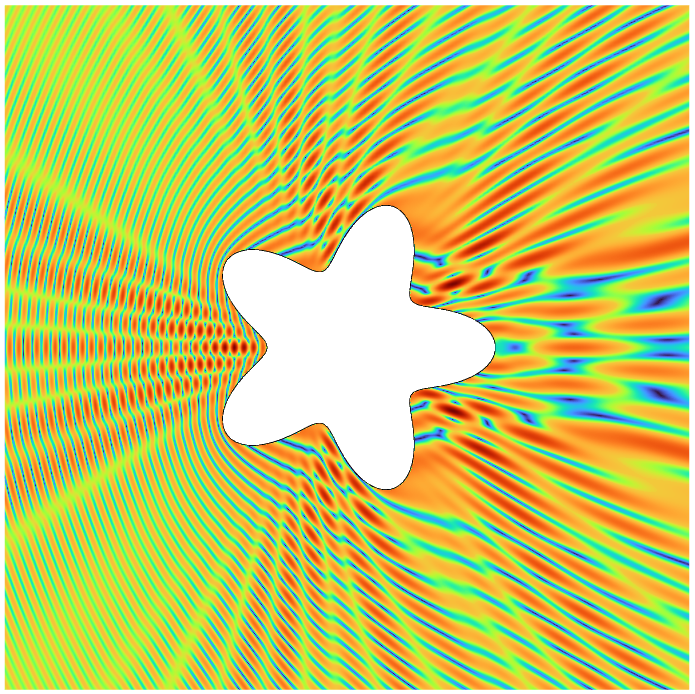} 
\caption{Absolute part of the total wave}
\label{fig:Star_Abs_TF}
\end{subfigure}
\begin{subfigure}{0.49\textwidth}
\includegraphics[width=0.85\linewidth]{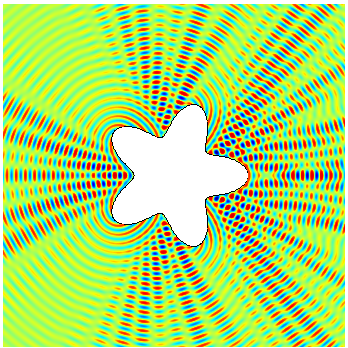}
\caption{Real part of the scattered wave}
\label{fig:Star_Real_SF}
\end{subfigure}
\caption{Illustration of the plane wave scattering by boundary of star shaped region for wavenumber $\kappa=48$.}
\vspace{-9mm}
\label{fig:Star_scat}
\end{center}
\end{figure}
\end{example}

\begin{example}
The parametrization of $\partial\Omega$ is 
$$c(t) = r(t)(\cos{t},\sin{t}),\quad r(t) = 1+0.3\cos{(4t+2\sin{t})}, ~0\leq t \leq 2\pi.$$
For the plane wave incidence with $\kappa = 10$, we compared our numerical solution with the reference solution computed at $21\times 21$ equispaced grids on $[-3,3]^2$. To compute the reference solution, we decomposed the jellyfish-shaped boundary into $P=512$ patches and considered $n=15$ on each patch. The GMRES took $23$ iterations with $10^{-13}$ tolerance for the computation of reference solution.
In \Cref{Table_Surf_JF1}, we display the high-order convergence in both the approaches, patches fixed and patches varying.
\Cref{fig:Jellyfish_scat}, shows the surface scattering by jellyfish-shaped boundary when the incoming plane wave impinges on the obstacle from the positive x-axis with $\kappa=50$. \
\begin{table}[hbt!]
	\begin{minipage}{0.5\linewidth}
\centering
\scalebox{0.8}{
						\begin{tabular}{c| c| c| c}
							\hline \hline
							\multirow{1}{*}{$P\times n$} & 
							\multirow{1}{*}{$\epsilon_{N,\infty}$} & 
							\multirow{1}{*}{noc} & 
							\multirow{1}{*}{time(in seconds)} \\ 
							\cline{1-4}
							\hline 
 $4\times 15$& 1.16e-01 & 2.20 & 0.02\\
$8\times 15$& 4.02e-03 & 4.85 & 0.06\\
 $16\times 15$& 1.77e-04 & 4.51 & 0.15\\
$32\times 15$ & 6.75e-08 & 11.4 & 0.41\\
 $64\times 15$& 2.14e-10 & 8.30 & 1.26\\
 $128\times 15$& 1.37e-13 & 10.6 & 4.28\\
 $256\times 15$& 4.70e-14 & 1.54 & 15.50\\
				\hline \hline
		\end{tabular} }
	\end{minipage}
	\begin{minipage}{0.5\linewidth}
			\centering
					\scalebox{0.8}{
						\begin{tabular}{c| c| c| c}
							\hline \hline
							\multirow{1}{*}{$P\times n$} & 
							\multirow{1}{*}{$\epsilon_{N,\infty}$} & 
							\multirow{1}{*}{noc} & 
							\multirow{1}{*}{time(in seconds)} \\ 
							\cline{1-4}
							\hline 
$10\times 4$ & 7.46 e-01 & - & -\\
$10\times 8$ & 5.63e-02 & 3.73 & 0.02\\
$10\times 16$ & 1.24e-03 & 5.50 & 0.05\\
$10\times 32$ & 2.43e-06 & 9.00 & 0.17\\
$10\times 64$ & 2.34e-10 & 13.3 & 0.76\\
$10\times 128$ & 1.37e-13 & 10.7 & 3.61\\
$10\times 256$ & 9.63e-14 & 0.51 & 20.8\\
				\hline \hline
		\end{tabular} } 
	\end{minipage}
	\caption{Illustration of high-order convergence of the proposed algorithm for the surface scattering by jellyfish-shaped scatterer with $\kappa = 10$.}
    \label{Table_Surf_JF1}
\end{table} 
\begin{figure}[H]
		\begin{subfigure}{0.49\textwidth}
			\includegraphics[width=0.85\linewidth]{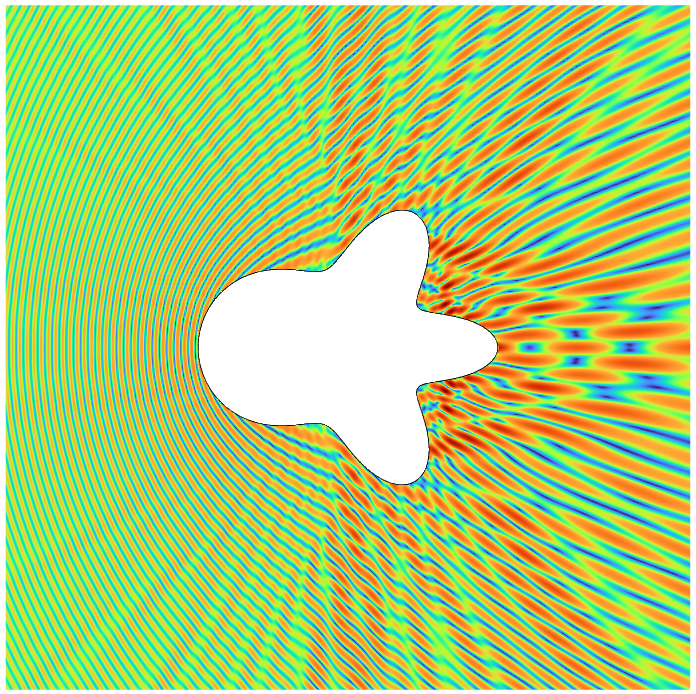} 
			\caption{Absolute value of the total wave}
			\label{fig:JF_absTotal}
		\end{subfigure}
		\begin{subfigure}{0.49\textwidth}
			\includegraphics[width=0.85\linewidth]{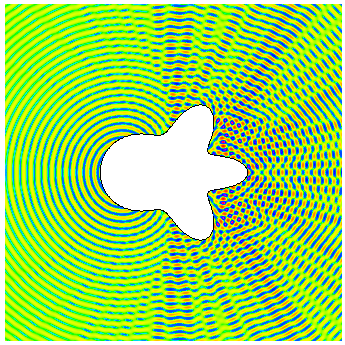}
			\caption{Real part of the scattered wave}
		\label{fig:JF_ReScat}
		\end{subfigure}
		\caption{Illustration of the plane wave scattering by boundary of jellyfish shaped region for wavenumber $\kappa=50$.}
        \vspace{-2mm}
		\label{fig:Jellyfish_scat}
\end{figure}    
\end{example}
\section{Conclusions}\label{conclusion}

We have presented an efficient high-order integration scheme for evaluating singular integral in one dimension. The proposed method is very accurate and can achieve machine accuracy with fewer integration nodes. We have established a complete quadrature analysis of the proposed integration scheme and verified the theoretical convergence rate through several computational results. The novel aspects of this contribution are the classification of all degrees of PCV (used in the analytical resolution of kernel singularity)  that can provide high-order convergence of the method and establishing a decay estimate on the continuous Chebyshev coefficients of integral density concerning its degree of smoothness and length of the domain. As an application of the method, we have employed the proposed scheme in conjunction with matrix-free iterative solver GMRES to solve impenetrable wave scattering by acoustic waves. 
In scattering simulations, the method's performance for large-scale problems is demonstrated through a variety of numerical results, and high-order convergence is obtained for a several complex domains.
\section{Acknowledgements}
Krishna acknowledges the financial support from CSIR throgh the file no. 09/1020(0183)/2019-EMR-I.
\bibliographystyle{abbrv} 
\bibliography{References.bib} 

\begin{thebibliography}{10}

\bibitem{apostol2013introduction}
T.~M. Apostol.
\newblock {\em Introduction to analytic number theory}.
\newblock Springer Science \& Business Media, 2013.

\bibitem{bruno2020chebyshev}
O.~P. Bruno and E.~Garza.
\newblock A {C}hebyshev-based rectangular-polar integral solver for scattering
  by geometries described by non-overlapping patches.
\newblock {\em Journal of Computational Physics}, 421:109740, 2020.

\bibitem{bruno2024direct}
O.~P. Bruno and A.~Pandey.
\newblock Direct/iterative hybrid solver for scattering by inhomogeneous media.
\newblock {\em SIAM Journal on Scientific Computing}, 46(2):A1298--A1326, 2024.

\bibitem{bruno_sachan2024numerical}
O.~P. Bruno and S.~Sachan.
\newblock A fractional {L}aplacian weakly singular integral solver.
\newblock in preparation.

\bibitem{colton2013integral}
D.~Colton and R.~Kress.
\newblock {\em Integral equation methods in scattering theory}.
\newblock SIAM, 2013.

\bibitem{criscuolo1990convergence}
G.~Criscuolo, G.~Mastroianni, and G.~Monegato.
\newblock Convergence properties of a class of product formulas for weakly
  singular integral equations.
\newblock {\em Math. Comp.}, 55(191):213--230, 1990.

\bibitem{davis2014methods}
P.~J. Davis and P.~Rabinowitz.
\newblock {\em Methods of Numerical Integration}.
\newblock Academic Press, 2014.

\bibitem{dominguez2013filon}
V.~Dom{\'i}nguez, I.~G. Graham, and T.~Kim.
\newblock {F}ilon--{C}lenshaw--{C}urtis rules for highly oscillatory integrals
  with algebraic singularities and stationary points.
\newblock {\em SIAM Journal on Numerical Analysis}, 51(3):1542--1566, 2013.

\bibitem{faria2021general}
L.~M. Faria, C.~P{\'e}rez-Arancibia, and M.~Bonnet.
\newblock General-purpose kernel regularization of boundary integral equations
  via density interpolation.
\newblock {\em Computer Methods in Applied Mechanics and Engineering},
  378:113703, 2021.

\bibitem{garza2022fast}
E.~Garza and C.~Sideris.
\newblock Fast inverse design of 3d nanophotonic devices using boundary
  integral methods.
\newblock {\em ACS Photonics}, 10(4):824--835, 2022.

\bibitem{hu2021chebyshev}
J.~Hu, E.~Garza, and C.~Sideris.
\newblock A {C}hebyshev-based high-order-accurate integral equation solver for
  {M}axwell’s equations.
\newblock {\em IEEE Transactions on Antennas and Propagation},
  69(9):5790--5800, 2021.

\bibitem{kress1990nystrom}
R.~Kress.
\newblock A {N}ystr{\"o}m method for boundary integral equations in domains
  with corners.
\newblock {\em Numerische Mathematik}, 58(1):145--161, 1990.

\bibitem{kutz1984asymptotic}
M.~K{\"u}tz.
\newblock Asymptotic error bounds for a class of interpolatory quadratures.
\newblock {\em SIAM journal on numerical analysis}, 21(1):167--175, 1984.

\bibitem{martinsson2007accelerated}
P.-G. Martinsson and V.~Rokhlin.
\newblock An accelerated kernel-independent fast multipole method in one
  dimension.
\newblock {\em SIAM Journal on Scientific Computing}, 29(3):1160--1178, 2007.

\bibitem{mason2002chebyshev}
J.~C. Mason and D.~C. Handscomb.
\newblock {\em Chebyshev polynomials}.
\newblock CRC press, 2002.

\bibitem{ng2023acoustic}
J.~Ng.
\newblock Acoustic waveform optimization for three-dimensional object
  geometries.
\newblock 2023.

\bibitem{trefethen2008gauss}
L.~N. Trefethen.
\newblock Is {G}auss quadrature better than {C}lenshaw--{C}urtis?
\newblock {\em SIAM review}, 50(1):67--87, 2008.

\bibitem{xiang2013convergence}
S.~Xiang.
\newblock On convergence rates of {F}ej{\'e}r and {G}auss--{C}hebyshev
  quadrature rules.
\newblock {\em Journal of Mathematical Analysis and Applications},
  405(2):687--699, 2013.

\bibitem{xiang2010error}
S.~Xiang, X.~Chen, and H.~Wang.
\newblock Error bounds for approximation in {C}hebyshev points.
\newblock {\em Numerische Mathematik}, 116(3):463--491, 2010.

\bibitem{xiang2015error}
S.~Xiang, G.~He, and Y.~J. Cho.
\newblock On error bounds of {F}ilon-{C}lenshaw-{C}urtis quadrature for highly
  oscillatory integrals.
\newblock {\em Advances in Computational Mathematics}, 41:573--597, 2015.

\bibitem{zhang2022hyper}
L.~Zhang, L.~Xu, and T.~Yin.
\newblock On the hyper-singular boundary integral equation methods for dynamic
  poroelasticity: three dimensional case.
\newblock {\em arXiv preprint arXiv:2202.04257}, 2022.

\bibitem{zhang2021fast}
Y.~Zhang, C.~Zhuang, and S.~Jiang.
\newblock Fast one-dimensional convolution with general kernels using
  sum-of-exponential approximation.
\newblock {\em Communications in Computational Physics}, 29(5):1570--1582,
  2021.

\end{thebibliography}
\end{document}